\documentclass[11pt,letterpaper]{amsart}
\usepackage{amsmath,amssymb}

\title{Radiation fields on Schwarzschild spacetime}
\author{Dean Baskin}
\address{Northwestern University}
\email{dbaskin@math.northwestern.edu}
\author{Fang Wang}
\address{Shanghai Jiao Tong University}
\email{fangwang1984@sjtu.edu.cn}

\date{January 9, 2014}

\thanks{The authors are grateful to Ant{\^o}nio S{\'a} Barreto for
  pointing out the simple proof of the $1+1$-dimensional support
  theorem.  D.B.~was supported by NSF postdoctoral fellowship
  DMS-1103436.} 

\usepackage{tikz}
\usepgflibrary{arrows}

\newtheorem{theorem}{Theorem}
\newtheorem{lemma}[theorem]{Lemma}
\newtheorem{proposition}[theorem]{Proposition}
\newtheorem{corollary}[theorem]{Corollary}

\newtheorem{definition}[theorem]{Definition}
\newtheorem{remark}[theorem]{Remark}

\numberwithin{theorem}{section}
\numberwithin{equation}{section}

\newcommand{\norm}[2][]{\left\| #2 \right\| _{#1}}
\newcommand{\pd}[1][]{\partial_{#1}}

\newcommand{\energyspace}{H_{E}}
\newcommand{\grad}{\nabla}
\newcommand{\horiz}{E_{1}^{+}}
\newcommand{\horizm}{E_{1}^{-}}
\newcommand{\lap}{\Delta}
\newcommand{\naturals}{\mathbb{N}}
\newcommand{\reals}{\mathbb{R}}
\newcommand{\scri}{S_{1}^{+}}
\newcommand{\scrim}{S_{1}^{-}}
\newcommand{\sphere}{\mathbb{S}}
\DeclareMathOperator{\supp}{supp}

\newcommand{\differential}[1]{\,d#1}
\newcommand{\da}{\differential{a}}
\newcommand{\dr}{\differential{r}}
\newcommand{\ds}{\differential{s}}
\newcommand{\dt}{\differential{t}}
\newcommand{\dT}{\differential{T}}
\newcommand{\dV}{\differential{V}}

\newcommand{\dmu}{\differential{\mu}}
\newcommand{\dnu}{\differential{\nu}}

\newcommand{\drho}{\differential{\rho}}
\newcommand{\dtau}{\differential{\tau}}

\newcommand{\domega}{\differential{\omega}}

\begin{document}

\begin{abstract}
  In this paper we define the radiation field for the wave equation on
  the Schwarzschild black hole spacetime.  In this context it has two
  components: the rescaled restriction of the time derivative of a
  solution to null infinity and to the event horizon.  In the process,
  we establish some regularity properties of solutions of the wave
  equation on the spacetime.  In particular, we prove that the
  regularity of the solution across the event horizon and across null
  infinity is determined by the regularity and decay rate of the
  initial data at the event horizon and at infinity.  We also show
  that the radiation field is unitary with respect to the conserved
  energy and prove support theorems for each piece of the radiation
  field.
\end{abstract}

\maketitle

\section{Introduction}
\label{sec:introduction}

In this paper we define the radiation field for the wave equation on
the Schwarzschild black hole spacetime.  The radiation field is a
rescaled restriction of the time derivative of a solution and in this
case has two components: one corresponding to null infinity and one
corresponding to the event horizon.  In the process,
we establish some regularity properties of solutions of the wave equation
on the spacetime.  In particular, we prove that the regularity of the
solution across the event horizon and across null infinity is
determined by the regularity and decay rate of the initial data at the event
horizon and at infinity.  We further show that the radiation field is
unitary with respect to the conserved energy and prove support
theorems for each component of the radiation field.

The radiation field for a solution of the wave equation describes the
radiation pattern seen by distant observers.  On Minkowski space
$\reals\times\reals^{n}$, it is the rescaled restriction of a solution
to null infinity.  More precisely, one introduces polar coordinates
$(r,\omega)$ in the spatial variables as well as the ``lapse''
parameter $s = t - r$.  The forward radiation field of a solution
$u$ of $(\pd[t]^{2}-\lap )u = 0$ with smooth, compactly supported
initial data is given by
\begin{equation*}
  \lim _{r\to \infty}\pd[s] r^{\frac{n-1}{2}}u (s+r, r\omega).
\end{equation*}
The map taking the initial data to the radiation field of the
corresponding solution provides a unitary isomorphism from the space
of finite energy initial data to
$L^{2}(\reals_{s}\times\sphere^{n-1}_{\omega})$.  The radiation field
is a translation representation of the wave group and was initially
defined by Friedlander~\cite{Friedlander:1980}, though it is implicit
in the work of Lax--Phillips (e.g.,~\cite{Lax:1989}) and Helgason
(e.g.,~\cite{Helgason:1999}).  Its definition, structure, and
properties have been studied in a variety of geometric
contexts~\cite{Sa-Barreto:2003,Sa-Barreto:2005,Sa-Barreto:2005a,Sa-Barreto:2008,BaskinBarreto2012},
including settings of interest in general
relativity~\cite{Wang:2011,BVW1}.

We now recall the structure of the Schwarzschild black hole spacetime.  (For a
more thorough discussion, including many different coordinate systems,
we direct the reader to the the book of Hawking and
Ellis~\cite{Hawking-Ellis} or to the lecture notes of Dafermos and
Rodnianski~\cite{Dafermos-Rodnianski:notes}.)  The Schwarzschild
spacetime is diffeomorphic to $\reals_{t}\times (2M, \infty)_{r}\times
\sphere^{2}_{\omega}$ with Lorentzian metric given by
\begin{equation*}
  g_{S} = -\left( \frac{r-2M}{r}\right) \dt^{2} + \left(
  \frac{r}{r-2M}\right) \dr^{2} + r^{2} \domega^{2}.
\end{equation*}
Here $\domega^{2}$ is the round metric on the unit sphere $\sphere^{2}$.

We consider the Cauchy problem:
\begin{align}
  \label{eq:IVP}
  &\Box_{S}u = 0, \quad(u,\pd[t]u)|_{t=0} = (\phi, \psi) 
\end{align}
where $\Box_{S}$ is the Laplace--Beltrami (D'Alembertian) operator for $g_{S}$:
\begin{equation*}
  \Box_{S} = - \left( \frac{r}{r-2M}\right)\pd[t]^{2} + \left(
    \frac{r-2M}{r}\right) \pd[r]^{2} + \frac{1}{r^{2}}\lap_{\omega} +
  \frac{2(r-M)}{r^{2}}\pd[r] .
\end{equation*}
Solutions $u$ of equation~\eqref{eq:IVP} possesses a
conserved energy $E(t)$:
\begin{equation*}
  E(t) = \int_{2M}^{\infty}\int_{\sphere^{2}}e(t)r^{2}\domega \dr 
\end{equation*}
where
\begin{equation*}
  e(t)= \left( 1 -
      \frac{2M}{r}\right)^{-1}(\pd[t]u)^{2} + \left( 1 -
      \frac{2M}{r}\right) (\pd[r]u)^{2} +
    \frac{1}{r^{2}}\left|\grad_{\omega}u\right|^{2}.
\end{equation*}
Observe that $e(t)$ is positive definite but is ill-behaved at $r=2M$.

As the Schwarzschild black hole has two spatial ends, there
are two ends through which null geodesics (i.e., light rays) can
``escape'': the event horizon (at the $r=2M$ end) and null infinity
(at the $r=\infty$ end).  (There are also ``trapped'' null geodesics
tangent to the photon sphere $r=3M$.)  In terms of incoming
Eddington--Finkelstein coordinates,
\[
(\tau=t + r + 2M\log (r-2M), r, \omega),
\]
the event horizon corresponds to $r=2M$.  Similarly, in terms of
outgoing Eddington--Finkelstein coordinates,
\[
(\bar{\tau}=t-r-2M\log (r-2M), r, \omega),
\]
null infinity corresponds to $r=\infty$.
Figure~\ref{fig:penrose-diag} depicts the Penrose diagram of the
Schwarzschild black hole exterior; $E_{1}^{\pm}$ corresponds to the
event horizons and $S_{1}^{\pm}$ to null infinity.

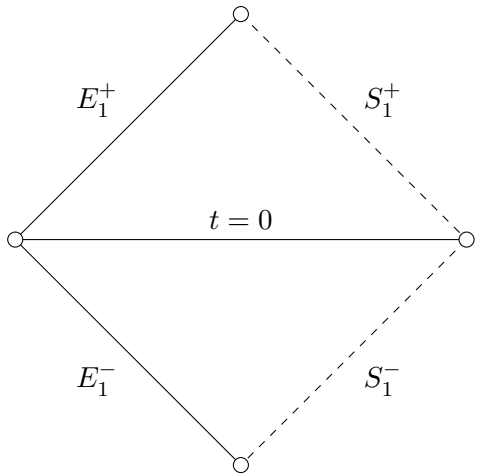
\begin{figure}[htp]
  \centering
  \begin{tikzpicture}
    \coordinate (left) at (0,0);
    \coordinate (top) at (3,3);
    \coordinate (right) at (6,0);
    \coordinate (bottom) at (3,-3);

    \draw [-] (top) -- (left) node [pos=0.5, anchor=south east] {$E_{1}^{+}$};
    \draw [-,dashed] (top) -- (right) node [pos=0.5, anchor = south
    west] {$S_{1}^{+}$};
    \draw [-,dashed] (right) -- (bottom) node[pos = 0.5, anchor =
    north west]{$S_{1}^{-}$};
    \draw [-] (bottom) -- (left) node[pos=0.5, anchor = north east]{$E_{1}^{-}$};
    \draw [-] (left) -- (right) node[pos=0.5, anchor = south]{$t=0$};

    \filldraw [fill=white, draw=black] (bottom) circle (0.1);
    \filldraw [fill=white, draw=black] (top) circle (0.1);
    \filldraw [fill=white, draw=black] (left) circle (0.1);
    \filldraw [fill=white, draw=black] (right) circle (0.1);
    
  \end{tikzpicture}
  \caption{The Penrose diagram of the Schwarzschild exterior.
    $E_{1}^{\pm}$ are the event horizons for the black and white
    holes, while $S_{1}^{\pm}$ are the future and past null infinities.}
  \label{fig:penrose-diag}
\end{figure}

In this paper we study the behavior of solutions to the Cauchy problem
in terms of their radiation fields.  We first partially compactify the
Schwarzschild spacetime to a manifold with corners and find expansions
for solutions of equation~\eqref{eq:IVP} at each boundary
hypersurface.  Two of the boundary hypersurfaces correspond to null
infinity and the event horizon.  Currently we do not include a full
compactification, as we omit temporal infinity.  A full
compactification on which solutions are well-behaved is expected to be
somewhat complicated, as it must take into account the different
expected behaviors of solutions at the event horizon, null infinity,
and the photon sphere ($r=3M$).

The main regularity result of this paper is the following:
\begin{theorem}
  \label{thm:regularity1}
  If the initial data $(\phi, \psi)$ has an asymptotic
  expansion at $r=2M$, i.e., if
  \begin{equation*}
    \phi=(r-2M)^{\lambda/2}\tilde{\phi}(\sqrt{r-2M}, \omega),\quad
    \psi=(r-2M)^{\lambda/2}\tilde{\psi}(\sqrt{r-2M}, \omega),
  \end{equation*}
  where $ \tilde{\phi},\tilde{\psi}\in C^{\infty}([0,\infty)\times\sphere^{2})$, 
  then the solution $u$ has an asymptotic expansion in terms of $e^{\tau/4M}$ at $\tau = -\infty$ near
  the event horizon and the $k$-th term (defined by equation~\eqref{kthterm.eventhorizon}) in the expansion is
  $C^{l,\delta}$ up to the event horizon, where $0 < \lambda + k = l +
  \alpha$ and $\delta = \min \{ \alpha - \epsilon, 1/2\}$ for
  any $\epsilon > 0$.

  Similarly, if the initial data has a classical asymptotic expansion
  at infinity, i.e., if
  \begin{equation*}
    \phi=r^{-\lambda-1}\tilde{\phi}(1/r, \omega),\quad
    \psi=r^{-\lambda-2}\tilde{\psi}(1/r, \omega),
  \end{equation*}
  where $ \tilde{\phi},\tilde{\psi}\in C^{\infty}([0,1/2M)\times\sphere^{2})$, 
  then the solution $u$
  has a (polyhomogeneous) asymptotic expansion in terms of $-1/\bar{\tau}$ at $\bar{\tau}=-\infty$ near null infinity and
  the $k$-th term (defined by equation~\eqref{kthterm.nullinfinity}) in the expansion is $C^{l,\delta}$ up to null infinity,
  where $l$ and $\delta$ are as above.
\end{theorem}

\begin{remark}
  \label{rem:partial-expansion}
  The above theorem shows that if the initial data are smooth and have
  asymptotic expansions at $r=2M$ and $r=\infty$, then the solution
  has an asymptotic expansion at $E_{0}$ and $S_{0}$ (as pictured in
  Figure~\ref{fig:partial-comp}).  The regularity at $E_{1}^{+}$
  (respectively, $S_{1}^{+}$) of each term in this expansion is
  determined by its rate of decay at $E_{0}$ (respectively, $S_{0}$)
  and hence by the expansion of the initial data.  The structure of
  the partial compactification is discussed in
  Section~\ref{sec:part-comp}.
\end{remark}

\begin{figure}[htp]
  \centering
  \begin{tikzpicture}
    \coordinate (ll) at (0,0);
    \coordinate (lr) at (6,0);
    \coordinate (ml) at (0,2); 
    \coordinate (mr) at (6,2);
    \coordinate (tl) at (2,4);
    \coordinate (tr) at (4,4);
    \coordinate (control) at (3,-0.25);
    \coordinate (a) at (0,1.5);
    \coordinate (b) at (0.5, 2.5);
    \coordinate (rho) at (1.5, 2.5);
    \coordinate (tau) at (1.5, 3.5);
    \coordinate (abar) at (6,1.5);
    \coordinate (bbar) at (5.5, 2.5);
    \coordinate (rhobar) at (4.5, 2.5);
    \coordinate (taubar) at (4.5, 3.5);

    \draw [->] (ml) -- (a) node[anchor=west]{$a$};
    \draw [->] (ml) -- (b) node[anchor=north west]{$b$};
    \draw [->] (1,3) -- (rho) node[anchor= north west]{$\rho$};
    \draw [->] (1,3) -- (tau) node[anchor = north west] {$\tau$};

    \draw [->] (mr) -- (abar) node[anchor = east] {$\bar{a}$};
    \draw [->] (mr) -- (bbar) node [anchor = north east]{$\bar{b}$};
    \draw[->] (5,3) -- (rhobar) node [anchor = north
    east]{$\bar{\rho}$};
    \draw [->] (5,3) -- (taubar) node [anchor = north east]{$\bar{\tau}$};

    \draw [-] (ll) node[anchor=north]{$r=2M$} .. controls (control) .. (lr)
    node[anchor=north]{$r=\infty$} node[pos=0.5, anchor=south]{$t=0$};
    \draw [-] (ll) -- (ml) node[pos=0.5,anchor=east]{$E_{0}$};
    \draw [-] (lr) -- (mr) node[pos=0.5,anchor=west]{$S_{0}$};
    \draw [-] (ml) -- (tl) node[pos=0.5, anchor = south
    east]{$E_{1}^{+}$};
    \draw [-] (mr) -- (tr) node[pos=0.5,anchor = south west]{$S_{1}^{+}$};

  \end{tikzpicture}
  \caption{The compactification of Schwarzschild spacetime for $t\geq
    0$: $E_1^+$ is the event horizon; $S_1^+$ is null infinity; $E_0$
    and $S_0$ are from the blow-up of the spatial ends in the Penrose
    diagram of the spacetime.}
  \label{fig:partial-comp}
\end{figure}
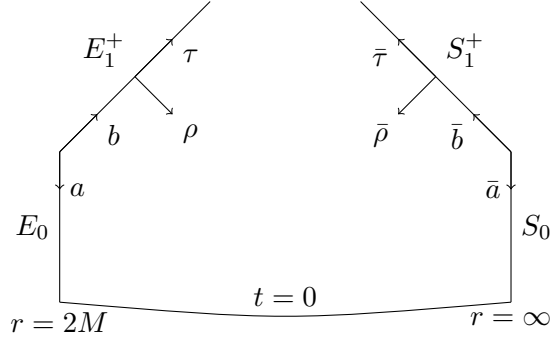

One consequence of Theorem~\ref{thm:regularity1} is the existence of
the radiation field, i.e., that solutions of equation~\eqref{eq:IVP}
may be restricted to the event horizon ($E_{1}^{+}$) and null infinity
($S_{1}^{+}$).  As there are two spatial ends of the Schwarzschild
black hole exterior, our definition of the radiation field has two
components.  For smooth initial data $(\phi,\psi)$ compactly supported
in $(2M, \infty) \times \sphere^{2}$, we define the two components of
the forward radiation field as follows:
\begin{align*}
 &\mathcal{R}_{\horiz}(\phi,\psi)(\tau,\omega)=\lim_{r\rightarrow 2M} \pd[t]u(\tau-r-2M\log(r-2M),r,\omega),\\
 &\mathcal{R}_{\scri}(\phi,\psi)(\bar{\tau},\omega)=\lim_{r\rightarrow \infty} r\pd[t]u(\bar{\tau}+r+2M\log(r-2M),r,\omega).
\end{align*}
The backward radiation fields $\mathcal{R}_{\horizm}$ and
$\mathcal{R}_{\scrim}$ are defined analogously in
equation~\eqref{eq:backward-rad-field}.

We show that, under our definition, the radiation
field is unitary (i.e., norm-preserving):
\begin{theorem}
  \label{thm:unitary1}
  Given $(\phi, \psi)$ with finite energy, the radiation field of the
  solution $u$ of equation~\eqref{eq:IVP} is unitary, i.e.,
  \begin{equation*}
    4M^{2}\norm[L^{2}(\reals \times
    \sphere^{2})]{\mathcal{R}_{\horiz}(\phi, \psi)}^{2} +
    \norm[L^{2}(\reals\times \sphere^{2})]{\mathcal{R}_{\scri}(\phi,
      \psi)}^{2} = E(0).
  \end{equation*}
\end{theorem}
The unitarity of this map strongly depends on the fact that the
Schwarzschild background is static.  As such, it is unlikely that that
the radiation field remains unitary under non-static metric
perturbations.  Within some classes of nonlinear problems, however, it
is reasonable to expect that the norm of the radiation field agrees
with the energy norm of the solution (see, e.g., recent work of the
first author and S{\'a} Barreto~\cite{BaskinBarreto2012}).

Bachelot~\cite{Bachelot:1994} (for the Klein-Gordon equation) and
Dimock~\cite{Dimock:1985} (for the wave equation) carried out related
work on the unitarity of wave operators in a more abstract
scattering-theoretic setting.

We also prove the following support theorem for the radiation field
(stated more precisely in Section~\ref{sec:supp-theor-radi}):
\begin{theorem}
  \label{thm:support-both}
  Suppose that $\phi,\psi \in C^{\infty}_{c}((2M, \infty) \times \sphere^{2})$.  If
  $\mathcal{R}_{\horiz}(\phi,\psi)$ vanishes for $\tau \leq \tau_{0}$ and
  $\mathcal{R}_{\horizm}(\phi,\psi)$ vanishes for $\tau \geq -\tau_{0}$,
  then both $\phi$ and $\psi$ are supported in $[r_{0},\infty)\times
  \sphere^{2}$, where $r_{0}$ is given implicitly by
  \begin{equation*}
    r_{0} + 2M \log (r_{0}-2M) = \tau_{0}.
  \end{equation*}
  An analogous statement holds for the component of the radiation field
  corresponding to null infinity.  
\end{theorem}

\begin{remark}
  The smoothness hypothesis above is not essential; we can relax the
  assumption to $C^{2,\alpha}$ for $\alpha > 0$.  In particular, we
  require only enough smoothness to ensure that the rescaled solution
  is $C^{2,\alpha}$ after being extended by zero across the event
  horizon or null infinity.  In fact, the support hypothesis near the
  event horizon can be relaxed as well; if the initial data has enough
  decay there (taking $\lambda > 2$ in Theorem~\ref{thm:regularity1}
  should suffice), then the rescaled solution will still have enough
  smoothness for the uniqueness theorems to apply.
  
  Near infinity, however, it is important that we take compactly
  supported data, as the past and future null infinities do not meet
  smoothly.  Even if this difficulty were overcome, a strong decay
  condition must be assumed to rule out the counter-examples that
  exist already in Minkowski space.  In particular, for any $m\in
  \mathbb{N}$, there are smooth functions $f(z)$ in $\reals^{n}$ which
  are not compactly supported, decay like $|z|^{-m}$, and whose Radon
  transform (and hence radiation field) is compactly supported (see,
  e.g., \cite{Helgason:1999}).
\end{remark}

In the setting of Minkowski space, the Fourier transform of the
forward radiation field in the $s$ variable is given in terms of the
Fourier transform of the initial data (and is related to the Radon
transform).  In other settings, one may think of the Fourier
transform of the radiation field as a distorted Fourier
transform.  The support theorem can then be seen as a
Paley--Wiener theorem for a distorted Fourier-type transform on
the Schwarzschild spacetime.  Similar support theorems in other
contexts have been established by Helgason~\cite{Helgason:1999},
S{\'a} Barreto~\cite{Sa-Barreto:2008,Sa-Barreto:2005}, and the first
author and S{\'a} Barreto~\cite{BaskinBarreto2012}.

When the wave evolution can be compared to a fixed background
evolution (as is the case for perturbations of Minkowski space), one
may define the M{\o}ller wave operator.  Given a solution of the
perturbed problem, one may often find a solution to the free problem
to which the solution of the perturbed problem scatters.  The
M{\o}ller wave operators associate the initial data for the solution
of the perturbed problem to the initial data for the solution of the
``free'' problem.  (We refer the reader to the book of Lax--Phillips
for more details \cite{Lax:1989}.)  In many such settings, the
radiation field and the M{\o}ller wave operators carry the same
information and thus the existence, unitary, and smoothness of the
radiation field are equivalent to the corresponding statements for the
M{\o}ller wave operator.  In fact, in those cases, it may even be
possible to obtain the support theorem from the M{\o}ller wave
operator.  We refer the reader to the recent paper of
Donninger--Krieger~\cite{Donninger:2013} for results in this
direction.  In the setting of the Schwarzschild background, however,
there is not an obvious natural evolution upon which to base the
M{\o}ller wave operator, and so the radiation field should be thought
of as a stand-in for this scattering-theoretic object.

Although we are unable to characterize the range of the radiation
field, we do show that our definition of the radiation field captures
``too much'' information.  In particular, the support theorem above
implies the following:
\begin{corollary}
  \label{thm:support1}
  If $\psi \in C^{\infty}_{c}\left( (2M, \infty) \times
    \sphere^{2}\right)$ and $\mathcal{R}_{\horiz}(0,\psi) \equiv 0$,
  then $\psi = 0$.  Similarly, for such a $\psi$, if
  $\mathcal{R}_{\scri}(0,\psi) \equiv 0$, then $\psi = 0$.
\end{corollary}
In other words, for odd, smooth, compactly supported data, knowing
that one component of the radiation field vanishes implies that it must
vanish on the other component as well.

The study of the asymptotic behavior of solutions of the wave equation
on the Schwarzschild background has been an active field of research
(see, e.g., the work of Dafermos--Rodnianski~\cite{Dafermos:2009a}, as
well as the works of Blue--Sterbenz~\cite{Blue:2006},
Blue--Soffer~\cite{Blue:2009}, Marzuola et al.~\cite{Marzuola:2010},
and Luk~\cite{Luk:2010}) and we refer the reader to the lecture notes
of Dafermos and Rodnianski~\cite{Dafermos-Rodnianski:notes} for many
references.  In Section~\ref{sec:extens-energy-space} we view the
unitarity of the radiation field as a consequence of the proofs of
Price's law (see the works of Tataru~\cite{Tataru:2013},
Metcalfe--Tataru--Tohaneanu~\cite{Metcalfe:2012}, and
Donninger--Schlag--Soffer~\cite{Donninger:2012}), though we remark
that a weaker energy decay result would suffice.

Section~\ref{sec:part-comp} describes the partial compactification on
which we work, while in Section~\ref{sec:radiation-field}, we prove
Theorem~\ref{thm:regularity1} via energy estimates.
Section~\ref{sec:extens-energy-space}
is devoted to the proof of Theorem~\ref{thm:unitary1} and relies on the
pointwise decay of solutions established by the proof of Price's law.
Finally, in Section~\ref{sec:supp-theor-radi}, we prove the support theorems.

\subsection{Notation}
\label{sec:notation}

In this section we lay out some notation common to the entire paper.

The regularity and asymptotic behavior of solutions near $E_{1}^{+}$
and $S_{1}^{+}$ may be obtained via energy estimates.  As a
preparation, for any function $v$ and time-like function $T$, we give
a name to the vector field obtained by contracting a modified
stress-energy tensor\footnote{We modify the stress-energy tensor by
  including a $v^{2}$ term in order to control inhomogeneous Sobolev
  norms.} of $v$ (with respect to a Lorentzian metric $g$) with the
gradient of $T$:
\begin{equation*}
  \mathcal{F}_{g}(T,v) = \langle \grad T, \grad v\rangle _{g} \grad v
  - \frac{1}{2} \left( \langle \grad v, \grad v\rangle_{g} + v^{2}
  \right) \grad T.
\end{equation*}
We also record
\begin{equation*}
  \operatorname{div}_{g}(\mathcal{F}_{g}(T,v)) = \langle \grad T,
  \grad v\rangle _{g} \left( \Box_{g} -1\right) v + \mathcal{Q}_{g}(T,v),
\end{equation*}
where
\begin{equation*}
  \mathcal{Q}_{g}(T,v) = \operatorname{Hess}_{g}(T)(dv,dv) -
  \frac{1}{2}\Box_{g}T\left( \langle \grad v, \grad v\rangle_{g} + v^{2}\right).
\end{equation*}

For a manifold with corners $M$, we further require the spaces of
uniformly degenerate ($0$-) and tangential ($b$-) vector fields:
\begin{align*}
  \mathcal{V}_{0} &= \left\{ X \in C^{\infty}(M , TM): X \text{
      vanishes at }\pd M\right\} \\
  \mathcal{V}_{b} &= \left\{ X \in C^{\infty}(M, TM): X \text{ is
      tangent to }\pd M\right\}
\end{align*}
For a given measure $\dmu$, the $0$-Sobolev space $H^{1}_{0}(M,\dmu)$
is the space of functions $u \in L^{2}(\dmu)$ so that $Xu \in
L^{2}(\dmu)$ for all $X\in \mathcal{V}_{0}$.  For an integer $N$ and a
given measure $\dmu$, the $b$-Sobolev space $H^{N}_{b}(M,d\mu)$
consists of functions $u \in L^{2}(\dmu)$ so that $X_{1}\ldots X_{k}u
\in L^{2}(\dmu)$ for $X_{j}\in \mathcal{V}_{b}$ and $k \leq N$.  The
mixed $0,b$-Sobolev space $H^{1,N}_{0,b}(M,\dmu)$ consists of those $u
\in H^{1}_{0}(M, \dmu)$ so that $X_{1}\ldots X_{k} u \in H^{1}_{0}(M,
\dmu)$ for $X_{j} \in \mathcal{V}_{b}$ and $k \leq N$.  For an
introduction to $b$-Sobolev spaces and $b$-geometry, we refer the
reader to the book of Melrose~\cite{Melrose:1993}.

\section{A partial compactification}
\label{sec:part-comp}

In this section we describe the partial compactification of
Schwarzschild spacetime for $t\geq 0$ (see Figure
\ref{fig:partial-comp}) on which we work. It can be obtained by a
suitable blow-up (with a logarithmic correction) of the Penrose
diagram in Figure \ref{fig:penrose-diag}. We describe the smooth
structure on this partial compactification by taking explicit local
coordinates near the boundary. Each set of local coordinates is valid
in the corresponding domain as in Figure~\ref{fig:partial-comp}.

We now describe coordinates giving the smooth structure on the partial
compactification (and the domains in which they are valid).
\begin{itemize}
\item Near $r=2M$ but away from temporal infinity, i.e., for
  $r-2M<\infty$ and $t+r+2M\log(r-2M)<\infty$, we choose coordinates
  $(a,b,\omega)$ with
   \[
   a=e^{-\frac{t+r}{2M}},\ b=e^{\frac{t+r}{4M}}\sqrt{r-2M}.
   \]
 \item Near the interior of the event horizon $\horiz$,
   i.e., $r-2M<\infty$ and $-\infty< t+r+2M\log(r-2M)<\infty$, we also
   can use coordinates $(\tau,\rho,\omega)$ with
   \[
   \tau=t+r+2M\log(r-2M),\ \rho=r-2M.
   \]
\end{itemize}
Here the initial surface $\{t=0\}$ for $r$ close to $2M$ is equivalent to 
$$\{a=e^{-\frac{2M+ab^2}{2M}}\}$$ 
for $a$ close to $e$ and $b$ close to $0$, 
which intersect with $E_0$ smoothly. 

For $r$ close to $\infty$, we choose corresponding coordinates as follows: 
\begin{itemize}
\item For $r$ large and close to the spatial end $S_0$, i.e., for
  $r-2M>0$ and $t-r-2M\log(r-2M)<0$, we choose coordinates
  $(\bar{a},\bar{b},\omega)$ with
  \[
  \bar{a}=\frac{-t+r+2M\log(r-2M)}{r}, \
  \bar{b}=\frac{1}{-t+r+2M\log(r-2M)}.
  \]
\item Near the interior of null infinity, i.e., $r-2M>0$ and $-\infty<
  t-r-2M\log(r-2M)<\infty$, we use coordinates
  $(\bar{\tau},\bar{\rho},\omega)$ with
  \[
  \bar{\tau}=t-r-2M\log(r-2M), \ \bar{\rho}=1/r.
  \]
\end{itemize}
Here the initial surface $\{t=0\}$ for $r$ large is equivalent to 
$$\bar{a}=1+2M\bar{a}\bar{b}\big(\log(1-2M\bar{a}\bar{b})-\log(\bar{a}\bar{b}) \big)\quad $$
for $\bar{a}$ close to $1$ and $\bar{b}$ close to $0$.  The initial surface does not
intersect $S_0$ smoothly but instead has a logarithmic correction
term. A direct consequence of this lack of smoothness is that for
classical initial data which have pure Taylor expansions at
$r=\infty$, the radiation field on $S_1^+$ will have an expansion
including logarithmic terms at $\bar{\tau}=-\infty$. We refer the
reader to Proposition \ref{prop:i4} and the surrounding discussion for
details.

\section{Existence and regularity of the radiation field}
\label{sec:radiation-field}

In this section we establish regularity properties for solutions of
equation~\eqref{eq:IVP} on a partial compactification of the
Schwarzschild background.  We also show that sufficiently regular
solutions have asymptotic expansions at the boundary hypersurfaces of
this compactification.  Taken together, Propositions~\ref{prop:e3} and
\ref{prop:i4} prove Theorem~\ref{thm:regularity1}.  

\subsection{At the event horizon}
\label{sec:event-horizon}

For $t>0$ we change coordinates to
\begin{equation*}
  \tau= t + r + 2M\log (r-2M),\ \rho = r-2M. 
\end{equation*}
The coordinates $(\tau,\rho)$ are essentially the incoming
Eddington-Finkelstein coordinates.  The function $\tau$ is a
coordinate along the event horizon $E^{+}_{1}$, while $\rho$ is a
defining function for $E_{+}^{1}$ as long as $\tau$ is bounded away
from $\pm\infty$.  
The metric and its D'Alembertian are then:
\begin{align*}
  &g_{S} = -\frac{\rho}{\rho + 2M}\dtau^{2} + 2\dtau\drho + (\rho +
  2M)^{2}\domega^{2},\\
 &\Box_{S} = 2\pd[\tau]\pd[{\rho}] + \frac{\rho}{\rho +
    2M}\pd[\rho]^{2} + \frac{1}{(\rho+2M)^{2}}\lap_{\omega} +
  \frac{2(\rho + M)}{(\rho + 2M)^{2}}\pd[\rho] + \frac{2}{\rho + 2M}\pd[\tau].
\end{align*}
We can thus extend $g_{S}$ naturally as a Lorentzian metric to a
slightly larger manifold (given in these coordinates by $\{\rho >
-\epsilon , |\tau | \leq C\}$) with $\{ \rho = 0\}$ a characteristic
hypersurface.  Moreover, if $Z_{ij}$ are the rotations of
$\sphere^{2}$, then
\begin{align*}
  &\left[ \Box_{S} , \pd[\tau] \right] =\left[ \Box_{S}, Z_{ij}\right] = 0,\\
  &\left[ \Box_{S},\rho \pd[\rho]\right] = \Box_{S} - \frac{3\rho +2M}{(\rho + 2M)^3}\lap_{\omega}+\frac{1}{(\rho+2M)^2}(\rho \pd[\rho])^{2} +V_1,\\
  &\left[ \Box_{S}, \pd[\rho]\right] = -\frac{2M}{(\rho+2M)^{2}}\pd[\rho]^{2}+\frac{2}{(\rho + 2M)^{3}}\lap_{\omega} +V_2,
\end{align*}
where $V_1, V_2$ are vector fields that are tangent to the event horizon:
\begin{align*}
 &V_1=  -\frac{4M}{(\rho+2M)^2}\pd[\tau]+ \frac{\rho - 2M}{(\rho + 2M)^3}(\rho\pd[\rho]),\\
 &V_2=  \frac{2}{(\rho + 2M)^{3}}(\rho\pd[\rho]) + \frac{2}{(\rho+2M)^{2}}\pd[\tau].
\end{align*}
In terms of the coordinates $(\tau, \rho,\omega)$, we choose a
time-like function by
\begin{equation*}
  T = \tau - \rho, \quad \langle \grad T, \grad T\rangle_{g_{S}} = -2 +
  \frac{\rho}{\rho + 2M} < -1.
\end{equation*}
We may then compute
\begin{align*}
  \langle \mathcal{F}(T,v), \grad T)\rangle_{g_{S}} & = 
  \frac{1}{2}\left( \left|\pd[\tau]v\right|^{2} + \left| \pd[\tau] v +
      \frac{\rho\pd[\rho]v}{\rho+2M}\right|^{2} \right) +
  \frac{2M\left| \pd[\rho] v\right|^{2}}{\rho + 2M} \\
  &\quad
  + \frac{\rho+4M}{2(\rho + 2M)}\left( \frac{\left|\grad_{\omega}v\right|^{2}}{(\rho + 2M)^{2}} + v^{2}\right), 
\end{align*}
\begin{align*}
 \mathcal{Q}(T,v) = -\frac{M}{(\rho + 2M)^{2}} \bigg( \left|
      \pd[\rho]v\right|^{2} + \frac{2\rho\left|\pd[\rho]v\right|^{2}}{\rho + 2M}      
       + 4 \pd[\rho]\pd[\tau]v - \frac{\left| \grad_{\omega}v\right|^{2}}{(\rho +
      2M)^{2}} + v^{2}\bigg).
\end{align*}

Using Friedlander's argument \cite{Friedlander:1980}, we now show that
solutions of the wave equation with compactly supported smooth initial
are smooth across the event horizon $E_{1}^{+}$.
\begin{proposition}
  \label{prop:e1}
  If $(\phi, \psi) \in C^{\infty}\left( (2M,\infty)_{r} \times
    \sphere^{2}\right)$ are such that $\supp (\phi)\cup \supp (\psi)
  \subset (2M+\epsilon, \infty) \times \sphere^{2}$ for some $\epsilon
  > 0$, then $u$ is smooth down to $\{ \rho = 0\}$ for all $\tau \in
  (-\infty, \infty)$.
\end{proposition}

\begin{figure}[htp]
  \centering
  \begin{tikzpicture}
    \coordinate (bl) at (0,0);
    \coordinate (top) at (6,6);
    \coordinate (br) at (12,0);
    \coordinate (mtau0) at (1,1);
    \coordinate (mtau0-0) at (2,0);
    \coordinate (ptau0) at (5,5);
    \coordinate (ptau0-0) at (10,0);
    \coordinate (control) at (6,3.5);
    \coordinate (mid) at (6,1);
    
    \draw [-] (bl) -- (top) node [pos = 0.5, anchor = south east]{$E_{1}^{+}$};
    \draw [-] (top) -- (br) node [pos = 0.5, anchor = south west]{$S_{1}^{+}$};
    \draw [-] (br) -- (bl) node [pos = 0.5, anchor = north] {$t=0$};
    \draw [-] (mtau0) -- (mtau0-0) node [pos = 0.5, anchor = west]{$\tau = -\tau_{0}$};
    \draw [-] (br) .. controls (control) ..(bl) node[pos =
    0.65, anchor = south]{$t=t_{0}$};
    \draw [-] (ptau0) -- (ptau0-0) node [pos=0.2, anchor = east]
    {$\tau = \tau_{0}$};
    
    \node [anchor=west] at (mid){$\mathbf{\Omega_{t_{0}}}$};
  \end{tikzpicture}
  \caption{The domain $\Omega_{t_0}$.}
  \label{fig:omega_t_0}
\end{figure}

\begin{proof}
  By finite speed of propagation, $u\equiv 0$ for $\tau \leq 2M +
  \epsilon + 2M\log \epsilon$.  We now fix $\tau_{0} > - (2M +
  \epsilon + 2M\log\epsilon)$ large.  Let $\Omega_{t_{0}}$ be the
  domain bounded by $\{ t = 0\}$, $\{ t = t_{0}\}$, $\{ \tau =
  \tau_{0}\}$, and $\{ \tau = -\tau _{0}\}$ (pictured in
  Figure~\ref{fig:omega_t_0}).  Here $\{ t=0\}$, $\{ t=t_{0}\}$ are
  space-like with defining function $t$ and $\{ \tau = \pm \tau
  _{0}\}$ are null with defining function $\tau$.  Moreover, $\langle
  \grad T , \grad t\rangle_{g_{S}} < 0$ and $\langle \grad T, \grad
  \tau \rangle _{g_{S}} < 0$ imply that
  \begin{equation*}
    \langle \mathcal{F}(T,v), \grad t \rangle_{g_{S}} \geq 0, \
    \langle \mathcal{F}(T,v), \grad \tau \rangle_{g_{S}} \geq 0.
  \end{equation*}
  Let $\Sigma_{s} = \{ T=s\}$ with defining function $T$ and
  $\Omega_{t_{0}}^{s} = \Omega_{t_{0}} \cap \{ T\leq s\}$.  Define now
  \begin{equation*}
    M^{N}(u,s) = \left( \sum _{ |I| \leq N} \int _{\Sigma_{s}\cap
        \Omega_{t_{0}}} e^{-cT}\langle
      \mathcal{F}\left(T,Z^{I}u\right), \grad T\rangle_{g_{S}} \dmu_{T}\right)^{1/2}
  \end{equation*}
  where $Z \in \{ \pd[\rho], \pd[\tau], Z_{ij}\}$ and $d\mu_{T}\wedge
  dT = dV_{g_{S}}$.  Note that each term in sum above is positive, so
  $M^{N}$ controls the first $N+1$ derivatives of $u$.  Choose $c$
  large enough so that
  \begin{align*}
   & \operatorname{div}\left( e^{-cT} \mathcal{F}(T,Z^{I}u)\right) \\
   &\quad\quad = e^{-cT}\left( -c \langle \mathcal{F}(T, Z^{I}u) + \langle \grad T,
    \grad Z^{I}u\rangle (\Box_{g_{S}}-1)Z^{I}u + Q_{g_{S}}(T,
    Z^{I}u)\right)\\
   &\quad\quad \leq 0
  \end{align*}
  for all $|I| \leq N$.  Here $c$ depends on $\tau_{0}$ and $N$.  By
  Stokes' theorem, we then have that
  \begin{equation*}
    \left(M^{N}(u,s)\right)^{2} \leq \sum_{|I| \leq N}\int_{\{t=0\}
      \cap \Omega_{t_{0}}} e^{-cT}\langle \mathcal{F}(T,Z^{I}u), \grad T\rangle \dmu_{t}.
  \end{equation*}
  Because $T$ is bounded on $\Omega_{\infty}$ and $\{ t=0\} \cap
  \Omega_{t_{0}}$ is independent of $t_{0}$, we have that
  \begin{equation*}
    \sum_{|I|\leq N}\int_{\Omega_{t_{0}}} e^{-cT} \langle
    \mathcal{F}(T,Z^{I}u), \grad T\rangle _{g_{S}}\dV_{g_{S}} =
    \int_{-\infty}^{\infty}\left( M^{N}(u,s)\right)^{2} \ds \leq
    C< \infty,
  \end{equation*}
  where $C$ depends only on $N$ and $\tau_{0}$.  Letting
  $t_{0}\to \infty$, we obtain
  \begin{equation*}
    \sum_{|I|\leq N}\int_{\Omega_{\infty}}e^{-cT}\langle \mathcal{F}(T, Z^{I}u), \grad
    T\rangle_{g_{S}}\dV_{g_{S}} \leq C_{N,\tau_{0}} .
  \end{equation*}
  Because $N$ and $\tau_{0}$ are arbitrary, $u$ is smooth up to $\{
  \rho = 0\}$ for all $\tau \in (-\infty, \infty)$.
\end{proof}

We now consider non-compactly supported data.  On the partial
compactification depicted in Figure~\ref{fig:partial-comp}, we use
coordinates $a$ and $b$, defined as follows:
\begin{equation*}
  a = \rho e^{-\frac{\tau}{2M}} = e^{-\frac{t+r}{2M}}\in [0,e^{-1}],\
  b = e^{\frac{\tau}{4M}} = e^{\frac{t+r}{4M}}\sqrt{r-2M}.
\end{equation*}
These are valid in a neighborhood of $E_0$, especially near the intersection of $E_{0}$ and
$E_{1}^{+}$, where $a$ is a defining function for $E_{1}^{+}$ and $b$
is a defining function for $E_{0}$.  (The function $r-2M = ab^{2}$
vanishes on both $E_{0}$ and $E_{1}^{+}$.) Near the interior
of $E_{1}^{+}$, $(a,b)$ are equivalent to the coordinates $(\tau,\rho)$ as
above. Near the initial surface $t=0$ and $r$ finite, $(a, b)$ are equivalent to the coordinates
$$(\beta,t)=(\sqrt{r-2M},t).$$
In coordinates $(a,b)$, the conformal metric $\tilde{g}_{e} =
(2b\sqrt{M})^{-2}g_{S}$ and its wave operator are:
\begin{align*}
  &\tilde{g}_{e} = 2\da\frac{\,db}{b} + 2a\left( 1 + \frac{ab^{2}}{2M +
      ab^{2}}\right)\left( \frac{\,db}{b}\right)^{2} + \frac{(2M +
    ab^{2})^{2}}{4Mb^{2}}\domega^{2}, \\
  &\tilde{\Box}_{e} = 2\pd[a]\left( b\pd[b] - a\pd[a] - 1\right) +
  \frac{4Mb^{2}}{(2M+ab^{2})^{2}}\lap_{\omega} -
  \frac{2b^{2}}{2M+ab^{2}}(a\pd[a])^{2} \\
  &\quad\quad\  - \frac{4b^{2}(M +
    ab^{2})}{(2M+ab^{2})^{2}}a\pd[a] + \frac{2b^{2}}{2M+ab^{2}}b\pd[b].
\end{align*}
We thus have that
\begin{align*}
  \Box_{g_{s}}u = 0 &\Longleftrightarrow \left( \tilde{\Box}_{e} +
    \gamma_{e}\right) \tilde{u} = 0, \\
  \tilde{u} = 2 b\sqrt{M} u , \ &\gamma_{e} = -b^{-1}\tilde{\Box}_{e}b = -
  \frac{2b^{2}}{2M + ab^{2}}.
\end{align*}
Moreover,
\begin{equation}
  \label{eq:commutators-with-e}
  \left[\tilde{\Box}_{e}, Z_{ij}\right] = 0, \ \left[
    \tilde{\Box}_{e},b\pd[b]\right] = b^{2}\sum_{|I|\leq
    2}c_{I}Z^{I},\ \left[ \tilde{\Box}_{e}, a\pd[a]\right] =
  \tilde{\Box}_{e} + b^{2}\sum_{|I|\leq 2}c_{I}'Z^{I},
\end{equation}
where $Z\in \{ a\pd[a], b\pd[b], Z_{ij}\}$and $c_{I},c_{I}'$ are
smooth coefficients.  In the above and what follows, we use $I$ as a
multi-index.  In coordinates $(a,b,\omega)$, we choose time-like
functions as follows:
\begin{align*}
  &T_{1} = -a + \log b,& &\langle \grad T_{1}, \grad
  T_{1}\rangle_{\tilde{g}_{e}} = -2 -2a\left( 1 + \frac{ab^{2}}{2M +
      ab^{2}}\right) < 0, \\
  &T_{1}' = -a,& &\langle \grad T_{1}' , \grad
  T_{1}'\rangle_{\tilde{g}_{e}} = -2a\left( 1 + \frac{ab^{2}}{2M +
      ab^{2}}\right) < 0.
\end{align*}
Here $T_{1}'$ is asymptotically null when approaching the event
horizon.  Moreover,  we again compute
\begin{align}
  \label{eq:calcs-near-event-horiz}
  &\langle \mathcal{F}_{\tilde{g}_{e}}(T_{1}, v), \grad T_{1}'\rangle_{\tilde{g}_{e}}  \\
  &\quad\quad = \left( 1 + \frac{ab^{2}}{2M+ab^{2}}\right)a\left| \pd[a]v\right|^{2}+\frac{1}{2} \left| b\pd[b]v\right|^{2} \notag\\
  &\quad\quad\quad+ \frac{1}{2} 
     \left|2\left( 1 + \frac{ab^{2}}{2M+ab^{2}}\right)a\pd[a]v -b\pd[b]v\right|^{2}  \notag\\
  &\quad\quad\quad
     + \frac{1}{2}\left( 1 +
    2a\left( 1 + \frac{ab^{2}}{2M + ab^{2}}\right)\right)\left(
    \frac{4M\left|b\grad_{\omega}v\right|^{2}}{(2M+ab^{2})^{2}}+v^{2} \right) , \notag
\\
  &\langle \mathcal{F}_{\tilde{g}_{e}}(T_{1},v), \grad \log  b\rangle_{\tilde{g}_{e}}\notag \\
  &\quad\quad  =\left|\pd[a] v\right|^{2}+ \left( 1 + \frac{ab^{2}}{2M+ ab^{2}}\right)a\left|\pd[a] v\right|^{2}  
      +\frac{1}{2}\left( \frac{4M\left|b\grad_{\omega}v\right|^{2}}{(2M+ab^{2})^{2}} + v^{2}\right), \notag\\
 &\mathcal{Q}_{\tilde{g}_{e}} (T_{1},v) \notag\\
 &\quad\quad
    = \left| \pd[a]v\right|^{2}-\left( 2 + \frac{2b^{2}}{2M+ab^{2}}\right)\left( \pd[a]v\right)\left(b\pd[b]v\right) 
   \notag\\
  &\quad\quad\quad
  +\left(4 +\frac{8Mb^{2}}{(2M+ab^{2})^{2}}\right)a\left| \pd[a]v\right|^{2} 
   + \Theta_{2}\left(a\pd[a]v,b\pd[b]v,b\nabla_{\omega}v,v\right) \notag,
\end{align}
where $\Theta_{2}(v^{1},\ldots,v^{l})$ is a quadratic form of
$(v^{1},\ldots,v^{l})$ with smooth coefficients.

\begin{proposition}
  \label{prop:e2}
  Suppose that $ (\phi, \psi)$ are in the following weighted Sobolev spaces
  \begin{equation*}
   \beta^{\lambda-1}H^{1,N}_{0,b}\left( [0,\beta_{0})_{\beta}\times
      \sphere^{2}, \frac{\differential{\beta}\domega}{\beta^{3}}\right) \times \beta^{\lambda
      -1}H^{1,N-1}_{0,b}\left( [0,\beta_{0})_{\beta} \times \sphere^{2}, \frac{\differential{\beta}\domega}{\beta^{3}}\right)
  \end{equation*}
  for some $\beta_0>0$, $N> 2$ and $\lambda > 0$.  Then $\tilde{u} = 2b\sqrt{M}u$ is $C^{\delta}$ up
  to $\{ a=0\}$ for $b < \beta_{0}e^{\frac{\beta_0^2+2M}{4M}}$, where $\delta = \min \{ \lambda, \frac{1}{2}\}$.
\end{proposition}

\begin{proof}
  The assumption on $\phi$ and $\psi$ implies that on the Cauchy surface 
  \begin{equation*}
   (Z^I\tilde{u})|_{t=0} \in \beta^{\lambda} H^{1,N-|I|}_{0,b}\left( [0,\beta_{0})\times
      \sphere^{2}, \frac{\differential{\beta}\domega}{\beta^{3}}\right)
  \end{equation*}
 for $Z\in \{a\partial_a, b\partial_b, Z^{ij}\}$.
  
  Let $\Sigma_{s} = \{ T_{1}'=s\}$, so that $\Sigma_{s}$ is space-like
  with defining function $T_{1}'$ for $s < 0$.  As $s\to 0$,
  $\Sigma_{s}$ approaches a characteristic surface.  Let $\Omega$ be
  the domain bounded by $\Sigma_{0}$, $\{t=0\}$, $\{b=0\}$ and $S = \{
  T_{1} = \log b_{0}\}$.  We then define
  \begin{align*}
   &M^{N}_{1}\left( \tilde{u},s;\lambda\right) = \left( \sum_{|I|\leq
      N}\int_{\Sigma_{s}\cap \Omega}b^{-2\lambda}e^{-cT_{1}'}\langle
    \mathcal{F}_{\tilde{g}_{e}}(T_{1},Z^{I}\tilde{u}), \grad
    T_{1}'\rangle_{\tilde{g}_{e}}\dmu_{T_{1}'}\right)^{\frac{1}{2}}
  ,\\
  &L^{N}_{1} \left( \tilde{u},s;\lambda\right) = \left( \sum_{|I|\leq
      N}\int_{S\cap
      \{T_{1}'<s\}\cap\Omega}b^{-2\lambda}e^{-cT_{1}'}\langle
    \mathcal{F}_{\tilde{g}_{e}}(T_{1},Z^{I}\tilde{u}),\grad
    T_{1}\rangle_{\tilde{g}_{e}}\dmu_{T_{1}}\right)^{\frac{1}{2}}, 
  \end{align*}
  where $Z\in \{ a\pd[a],b\pd[b],Z_{ij}\}$ and $\dmu_{T_{1}'}\wedge
  \dT_{1}' = \dV_{\tilde{g}_{e}}$, $\dmu_{T_{1}} \wedge \dT_{1} =
  \dV_{\tilde{g}_{e}}$.

  We first choose $s_{0}$ close to $0$ so that $\Sigma_{s_{0}}\cap
  \{t=0\} \cap \Omega=\emptyset$.  Let $\Omega_{s_{0}} = \Omega\cap
  \{s\leq s_{0}\}$ (illustrated in Figure~\ref{fig:omega-s-0}) so that
  $\Omega_{s_{0}}$ is bounded by $\Sigma_{s_{0}}$, $\{ t= 0\}$,
  $\{b=0\}$, and $S$.  Choose $c$ large enough so that in
  $\Omega_{s_{0}}$ we have
  \begin{align*}
    &\sum_{|I|\leq N} \operatorname{div}_{\tilde{g}_{e}}\left(
      b^{-2\lambda}e^{-cT'}\mathcal{F}_{\tilde{g}_{e}}(T_{1},Z^{I}\tilde{u})\right) \\
   &\quad\quad= b^{-2\lambda}e^{ - cT_{1}'}\sum_{|I|\leq N} \bigg(
    -2\lambda \langle
    \mathcal{F}_{\tilde{g}_{e}}(T_{1},Z^{I}\tilde{u}), \grad \log
    b\rangle_{\tilde{g}_{e}}  + \mathcal{Q}_{\tilde{g}_{e}}(T_{1},
    Z^{I}\tilde{u}) \\
   &\quad\quad\quad - c\langle \mathcal{F}_{\tilde{g}_{e}}
    (T_{1}, Z^{I}\tilde{u}), \grad T_{1}'\rangle _{\tilde{g}_{e}}+ \langle \grad T_{1} , \grad
    Z^{I}\tilde{u}\rangle_{\tilde{g}_{e}} \left( \tilde{\Box}_{e}
      -1\right)Z^{I}\tilde{u} \bigg) \leq 0.
  \end{align*}
  Here $c$ depends only on $b_{0},s_{0}$, and $N$, and the dependence
  on $s_{0}$ is required to bound $|\pd[a]Z^{I}\tilde{u}|^{2}$ in
  terms of $a|\pd[a]Z^{I}\tilde{u}|^{2}$.  In order
  to bound the term of the form
  \begin{equation*}
    \langle \grad T_{1}, \grad
    Z^{I}\tilde{u}\rangle_{\tilde{g}_{e}}\left(
      \tilde{\Box}_{e}-1\right)Z^{I}\tilde{u}, 
  \end{equation*}
  we have used that $\tilde{u}$ solves $\left( \tilde{\Box}_{e} +
    \gamma_{e}\right) \tilde{u} = 0$ and the
  expressions~(\ref{eq:commutators-with-e}) for the commutators of
  $\tilde{\Box}$ with $Z^{I}$.  By Stokes' theorem, we then have
  \begin{align*}
    &\left( M_{1}^{N}(\tilde{u},s_{0})\right)^{2} + \left(
      L_{1}^{N}(\tilde{u},s_{0})\right)^{2} \\
    &\quad\quad \leq \sum_{|I|\leq N}
               \int _{\{t=0\} \cap \Omega} b^{-2\lambda}e^{ - cT_{1}'}\langle
    \mathcal{F}_{\tilde{g}_{e}}(T_{1}Z^{I}\tilde{u}), \grad t\rangle _{\tilde{g}_{e}}\dmu_{t},
  \end{align*}
  where $\dmu_{t} \wedge \dt = \dV_{\tilde{g}_{e}}$.  Observe that the
  right hand side is equivalent to the square of the initial data
  norm.

  \begin{figure}[htp]
    \centering
    \begin{tikzpicture}
      \coordinate (ll) at (0,0);
      \coordinate (lr) at (6,0);
      \coordinate (ml) at (0,2); 
      \coordinate (mr) at (6,2);
      \coordinate (tl) at (2,4);
      \coordinate (tr) at (4,4);

      \coordinate (St) at (1,3);
      \coordinate (Sb) at (5,0);
      \coordinate (Scont) at (3,2);
      \coordinate (Sigmal) at (0,1);
      \coordinate (Sigmar) at (4,1);
      \coordinate (Sigmacont) at (2,1.25);

      \coordinate (omegatext) at (2,0.5);

      \draw [-] (St) .. controls (Scont) .. (Sb) node [pos = 0.5, anchor
      = south west]{$S$};

      \draw[-] (Sigmal) .. controls (Sigmacont) .. (Sigmar) node
      [pos=0.5, anchor = south]{$\Sigma_{s_{0}}$};

      \draw [-] (ll) node[anchor=north]{$r=2M$} -- (lr)
      node[pos = 0.5, anchor=north]{$t=0$};
      \draw [-] (ll) -- (ml) node[pos=0.5,anchor=east]{$E_{0}$};
      \draw [-] (ml) -- (tl) node[pos=0.5, anchor = south
      east]{$E_{1}^{+}$};

      \node at (omegatext){$\Omega_{s_{0}}$};
    \end{tikzpicture}
    \caption{The domain $\Omega_{s_0}$}
    \label{fig:omega-s-0}
  \end{figure}
 
  For $s > s_{0}$, we now let $\Omega_{s_{0}}^{s}$ denote the domain
  bounded by $\Sigma_{s_{0}}$, $\Sigma_{s}$, $\{ b= 0\}$, and $S$.
  Again Stokes' theorem implies that
  \begin{align*}
    &\left( M_{1}^{N}(\tilde{u},s)\right)^{2} - \left(
      M_{1}^{N}(\tilde{u},s_{0})\right)^{2} + \left(
      L_{1}^{N}(\tilde{u},s)\right)^{2} - \left(
      L_{1}^{N}(\tilde{u},s_{0})\right)^{2} \\
   &\quad\quad = \sum_{|I|\leq N}
    \int_{\Omega_{s_{0}}^{s}}\operatorname{div}\left( b^{-2\lambda}
    e^{ - cT_{1}'}\mathcal{F}_{\tilde{g}_{e}}(T_{1}, Z^{I}\tilde{u})\right)\dV_{\tilde{g}_{e}}.
  \end{align*}
  Dividing by $s-s_{0}$, taking a limit, and then using the
  expressions~\eqref{eq:calcs-near-event-horiz} yields that
  \begin{align*}
    &\pd[s]\left( M_{1}^{N}(\tilde{u},s)\right)^{2} + \pd[s]\left(
      L_{1}^{N}(\tilde{u},s)\right)^{2} \\
    &\quad\quad \leq
    \begin{cases}
      \left( (1-2\lambda)a^{-1} + Ca^{-\frac{1}{2}}\right)\left(
        M_{1}^{N}(\tilde{u},s)\right)^{2} & \lambda < \frac{1}{2} \\
      C a^{-\frac{1}{2}}\left( M_{1}^{N}(\tilde{u},s)\right)^{2} &
      \lambda \geq \frac{1}{2}
    \end{cases},
  \end{align*}
  with $C$ a constant depending only on $b_{0}$ and $N$.  (Note that
  the power $a^{-1/2}$ arises from the need to estimate the terms of
  the form $(\pd[a]Z^{I}\tilde{u})(b\pd[b]Z^{I}\tilde{u})$.)
  Integrating in $s$, we find that for all $s > s_{0}$,
  \begin{equation*}
    M_{1}^{N}(\tilde{u},s) \leq
    \begin{cases}
      C' (-s)^{\lambda -
        \frac{1}{2}}(-s_{0})^{\frac{1}{2}-\lambda}M_{1}^{N}(\tilde{u},s_{0})
      & \lambda < \frac{1}{2} \\
      C' M_{1}^{N}(\tilde{u},s_{0}) & \lambda \geq \frac{1}{2}
    \end{cases},
  \end{equation*}
  where $C' = \exp \left( \frac{1}{2}
    \int_{0}^{e^{-1}}Ca^{-1/2}\da\right)$.  Because $N > 2$, the Sobolev embedding
  theorem implies that for $a < -s_{0}$, we have
  \begin{equation*}
    \left| \pd[a] \tilde{u}\right| \leq
    \begin{cases}
      C'' M_{1}^{N}(\tilde{u},s_{0})a^{\lambda -1} & \lambda <
      \frac{1}{2} \\
      C'' M_{1}^{N}(\tilde{u},s_{0})a^{-\frac{1}{2}} & \lambda \geq \frac{1}{2}
    \end{cases},
  \end{equation*}
  where $C''$ only depends on $b_{0}$.  Integrating in $a$ then finishes the proof.
\end{proof}

\begin{corollary}
  \label{cor:e1}
  Suppose $\lambda = k + \alpha$ with $\alpha \in (0,1]$ and $k\in
  \naturals_{0}$.  If the pair $(\phi, \psi)$ lies in the following space
  \begin{equation*}
     \beta^{\lambda - 1}H^{1,N}_{0,b}\left(
      [0,\beta_{0})\times \sphere^{2} ,
      \frac{\differential{\beta}\domega}{\beta^{3}}\right) \times
    \beta^{\lambda-1}H^{1,N-1}_{0,b}\left( [0,\beta_{0})\times \sphere^{2}, \frac{\differential{\beta}\domega}{\beta^{3}}\right)
  \end{equation*}
  with $N > 2 + k$, then $\tilde{u}$ is $C^{k,\delta}$ up to $\{a=0\}$
  for $b < \beta_{0}e^{\frac{\beta_0^2+2M}{4M}}$, where $\delta = \min \{ \alpha , \frac{1}{2}\}$.
\end{corollary}

\begin{proof}
  Notice that
  \begin{align*}
    &\left[ \tilde{\Box}_{e}, \pd[a]^{k}\right]  = 2k\pd[a]^{k+1} +
    b^{2} \sum_{|I|+i\leq k+1,  i \leq k}c_{I,i}Z^{I}\pd[a]^{i}, \\
    &\left[ a\pd[a], \pd[a]^{k}\right] = -k \pd[a]^{k} , \quad \left[ b\pd[b], \pd[a]^{k}\right] = \left[ Z_{ij},
      \pd[a]^{k}\right] = 0,
  \end{align*}
  where $Z\in \{ a\pd[a], b\pd[b], Z_{ij}\}$ and $c_{I,i}$ are smooth
  coefficients.  We set
  \begin{align*}
    &\widetilde{M}_{1}^{N} (\tilde{u}, s; \lambda) = \left(
      \sum_{i\leq k} \left( M_{1}^{N-i}(\pd[a]^{i}\tilde{u}, s;\lambda)\right)^{2}
    \right)^{\frac{1}{2}}, \\
    &\widetilde{L}_{1}^{N}(\tilde{u},s;\lambda) = \left(\sum_{i\leq
        k}\left( L_{1}^{N-i}(\pd[a]^{i}\tilde{u},s;\lambda)\right)^{2}
    \right)^{\frac{1}{2}} .
  \end{align*}
  According to the divergence formula in the proof of
  Proposition~\ref{prop:e2}, we have that
  \begin{align*}
    &\pd[s] \left( \widetilde{M}_{1}^{N}(\tilde{u},s)\right)^{2} +
    \pd[s]\left( \widetilde{L}_{1}^{N}(\tilde{u},s)\right)^{2} \\
   &\quad\quad\leq
    \begin{cases}
      \left( (1-2\alpha)a^{-1} + Ca^{-\frac{1}{2}}\right)\left(
        \widetilde{M}_{1}^{N}(\tilde{u},s)\right)^{2} & \alpha <
      \frac{1}{2} \\
      Ca^{-\frac{1}{2}}\left(
        \widetilde{M}_{1}^{N}(\tilde{u},s)\right)^{2} & \alpha \geq \frac{1}{2}
    \end{cases}.
  \end{align*}
  We thus obtain the following in a similar manner to the proof of
  Proposition~\ref{prop:e2}:
  \begin{align*}
   & \left| \pd[a]^{i}\tilde{u}\right|  \leq
    C'\widetilde{M}_{1}^{N}(\tilde{u},s_{0}), \  a < -s_{0}, i \leq  k ; \\
   & \left| \pd[a]^{k+1}\tilde{u}\right|  \leq
    \begin{cases}
      C' \widetilde{M}_{1}^{N}(\tilde{u},s_{0})a^{\alpha -1} & \alpha
      < \frac{1}{2} \\
      C' \widetilde{M}_{1}^{N}(\tilde{u},s_{0})a^{-\frac{1}{2}} &
      \alpha \geq \frac{1}{2} 
    \end{cases},
  \end{align*}
  which finishes the proof.  
\end{proof}

In particular, ``Schwartz'' initial data behaves in much the same way
that compactly supported data does:
\begin{corollary}
  If $(\phi, \psi)$ lies in the space
  \begin{equation*}
 \beta^{\infty} H^{1,\infty}_{0,b}\left([0,\beta_0) \times
    \sphere^{2}, \frac{\differential{\beta}\domega}{\beta^{3}}\right) \times
  \beta^{\infty} H^{1,\infty}_{0,b}\left( [0,\beta_{0}) \times \sphere^{2} ,
    \frac{\differential{\beta}\domega}{\beta^{3}}\right), 
  \end{equation*}
  then $\tilde{u}$ is smooth up to $\{ a = 0\}$ for $b < \beta_{0}e^{\frac{\beta_0^2+2M}{4M}}$.
\end{corollary}

We now claim that $u$ has a classical asymptotic expansion if the
initial data does.  Indeed, suppose that
\begin{equation*}
  \phi, \psi \in \beta^{\lambda} C^{\infty}([0,\beta_{0})_{\beta}\times \sphere^{2}),
\end{equation*}
and define the ``k-th term'' $w^{k}$ by 
\begin{equation}
\label{kthterm.eventhorizon}
w^{0} = \tilde{u} = 2b\sqrt{M}u, \quad
w^{k} = \Pi_{i=1}^{k}\left( b\pd[b] - \lambda - i\right)\tilde{u} \quad \mathrm{for}\quad k\geq 1.
\end{equation}
 Notice that $b\partial b$ lifts to $\beta\partial_\beta-2\beta^2\partial_t$ on initial surface 
$t=0$. Hence 
$$(Z^Iw^k)|_{t=0}\in \beta^{\lambda+k} C^{\infty}([0,\beta_{0})_{\beta}\times \sphere^{2}) $$ 
for $Z\in\{a\partial_a, b\partial_b, Z^{ij}\}$ and multi-index $I$. 

\begin{proposition}
  \label{prop:e3}
  If $\phi$ and $\psi$ are as above with $0 < \lambda + k = l +
  \alpha$ for some integer $l$ and $\alpha \in (0,1]$, then $w^{k}$ is
  $C^{l,\delta}$ up to $\{ a= 0\}$ for $b < b_{0}$, where $\delta =
  \min \{ \alpha -\epsilon, \frac{1}{2}\}$ with $\epsilon > 0$
  arbitrarily small.
\end{proposition}

\begin{proof}
  Notice that because $\left( \tilde{\Box}_{e}
    +\gamma_{e}\right)\tilde{u}=0$, we have 
  \begin{align*}
    \tilde{\Box}_{e} w^{i} &= b^{2}\sum_{|I|\leq 2}c_{I}Z^{I}w^{i-1} +
    (b\pd[b] + \lambda - i)\tilde{\Box}_{e}w^{i-1} \\
    &= b^{2} \sum_{j=0}^{i}\sum_{|I|\leq 2} c_{I,j}Z^{I,j}w^{j},
  \end{align*}
  where $Z \in \{ a\pd[a], b\pd[b], Z_{ij}\}$ and $c_{I,j}$ are smooth
  coefficients.  We choose $0 < \epsilon_{0} < \epsilon_{1}< \ldots <
  \epsilon_{k}=\epsilon$ and let $\lambda_{i} = \lambda + i -
  \epsilon_i$.  By Proposition~\ref{prop:e2}, choosing $N$ large and
  then $c = c(s_{0},b_{0},N)$, we may guarantee
  \begin{align*}
    &M_{1}^{N}(w^{i},s_{0};\lambda_{i}) \\
    &\quad\quad \leq 
    \sum_{|I|\leq N}\int_{\{t=0\}\cap\Omega}b^{-2\lambda_{i}}e^{-cT'}\langle
    \mathcal{F}_{\tilde{g}_{e}}(T_{1},Z^{I}w^{i}),\grad t\rangle
    _{\tilde{g}_{e}}\dmu_{t} \leq
     C_{i} < \infty 
  \end{align*}
  for all $i \leq k$.  Assume now that
  \begin{equation}
    \label{eq:w-i-bounded}
    M_{1}^{N-2i}\left(w^{i},s;\lambda_{i}\right) \leq C_{i}' \left( 1
      + (-s)^{\lambda_{i}-\frac{1}{2}}\right), 
  \end{equation}
  for $s > s_{0}$ and $i \leq k -1$.  This then implies that 
  \begin{equation*}
    \left( \int_{\Sigma_{s}\cap \Omega}b^{-2\lambda_{i+1}}e^{-cT'}\left(
        b^{2} \sum_{|I|\leq
          N-2i}c_{I}Z^{I}w^{i}\right)^{2}\dmu_{t}\right)^{\frac{1}{2}}
    \leq C_{i}''\left( 1 + (-s)^{\lambda_{i}}\right).
  \end{equation*}
  Then for $\lambda_{i+1}\leq \frac{1}{2}$, we have (as in the proof
  of Proposition~\ref{prop:e2})
  \begin{align*}
    &\pd[s]\left( M_{1}^{N-2i-2}\left( w^{i+1},
        s;\lambda_{i+1}\right)\right)^{2} \\
    &\quad\quad\leq \left( (1-2\lambda_{i+1})a^{-1} + Ca^{-\frac{1}{2}}\right)\left(
      M_{1}^{N-2i-2}(w^{i+1},s;\lambda_{i+1})\right)^{2} \\
    &\quad\quad\quad + Ca^{-\frac{1}{2}}M_{1}^{N-2i-2}
    (w^{i+1},s;\lambda_{i+1})(1+(-s)^{\lambda_{i}}), 
  \end{align*}
  while if $\lambda_{i+1}\geq \frac{1}{2}$, we have
  \begin{align*}
    &\pd[s]\left(
      M_{1}^{N-2i-2}(w^{i+1},s;\lambda_{i+1})\right)^{2}\\
    &\quad\quad \leq Ca^{-\frac{1}{2}} \left(
    M_{1}^{N-2i-2}(w^{i+1},s;\lambda_{i+1})\right)^{2} \\
   &\quad\quad\quad + Ca^{-\frac{1}{2}}M_{1}^{N-2i-2}(w^{i+1},s;\lambda_{i+1})\left( 1 +
    (-s)^{\lambda + k - 1 - \epsilon_{k-1}}\right).
  \end{align*}
  Hence,
  \begin{equation*}
    M_{1}^{N-2i-2}(w^{k},s;\lambda_{i+1})\leq C_{i+1}' \left( 1 +
      (-s)^{\lambda_{i+1}-\frac{1}{2}}\right) 
  \end{equation*}
  for all $s > s_{0}$.  Now we know that
  equation~(\ref{eq:w-i-bounded}) holds for $i=0$.  By induction and
  the arbitrariness of $\epsilon$, we have thus shown that
  \begin{equation*}
    M_{1}^{N-2k}\left( w^{k},s;\lambda+k-\epsilon \right) \leq C_{k}'
    \left( 1 + (-s)^{\lambda + k - \epsilon - \frac{1}{2}}\right)
  \end{equation*}
  for all $s > s_{0}$, where $C_{k}'$ depends on the initial data and
  $\epsilon$, $N$, and $k$.  We thus have that if $\lambda + k = l +
  \alpha > 0$, then by inserting $\pd[a]^{l}$ as in the proof of
  Corollary~\ref{cor:e1}, we see that for $\delta = \min \{ \alpha -
  \epsilon, \frac{1}{2}\}$, $w^{k}$ is $C^{l,\delta}$ up to
  $\{a=0\}$.
\end{proof}

\subsection{At future null infinity}
\label{sec:near-future-null}

We now adapt the argument from Section~\ref{sec:event-horizon} to the
region near null infinity.  For $t>0$ we introduce the coordinates
\begin{equation*}
 \bar{ \tau} = t - r - 2M\log (r-2M), \quad \bar{\rho} = \frac{1}{r}.
\end{equation*}
The conformal metric and its wave operator then become:
\begin{align*}
  &\tilde{g}_{\infty} = \bar{\rho}^{2}g_{S} = -(1-2M\bar{\rho})\bar{\rho}^{2}\dtau^{2} +
  2d\bar{\tau}d\bar{\rho}+ \domega^{2}, \\
  &\tilde{\Box}_{\infty} = 2\pd[\bar{\tau}]\pd[\bar{\rho}] +
  (1-2M\bar{\rho})(\bar{\rho}\pd[\bar{\rho}])^{2} + (1-4M\bar{\rho})\bar{\rho}\pd[\bar{\rho}] + \lap_{\omega}.
\end{align*}
We now calculate the commutators of $\tilde{\Box}_{\infty}$ with
various vector fields:
\begin{align*}
  &\left[ \tilde{\Box}_{\infty}, \pd[\bar{\tau}]\right] = \left[
    \tilde{\Box}_{\infty}, Z_{ij}\right] = 0, \\
  &\left[ \tilde{\Box}_{\infty}, \bar{\rho} \pd[\bar{\rho}] \right] =
  \tilde{\Box}_{\infty} - \left( 1 - 4M\bar{\rho}\right)(\bar{\rho}\pd[\bar{\rho}])^{2}
  - \left( 1- 8M\bar{\rho}\right)(\bar{\rho} \pd[\bar{\rho}]) - \lap_{\omega}, \\
  &\left[ \tilde{\Box}_{\infty}, \pd[\bar{\rho}]\right] = -2\left( 1 -
    3M\bar{\rho}\right)\bar{\rho}\pd[\bar{\rho}]^{2} - 2\left( 1 - 6M\bar{\rho}\right)\pd[\bar{\rho}].
\end{align*}
Moreover, we have that
\begin{align*}
  &\Box_{S} u =0 \quad \Longleftrightarrow \quad \left(
    \tilde{\Box}_{\infty} + \gamma_{\infty}\right) \tilde{u} = 0, \\
  &\tilde{u}= \bar{\rho}^{-1} u, \quad \gamma_{\infty} = - \bar{\rho}
  \tilde{\Box}_{\infty}\bar{\rho}^{-1} = -2M\bar{\rho}.
\end{align*}
In the coordinates $(\bar{\tau}, \bar{\rho}, \omega)$, we choose the time-like
function for $\bar{\rho} < \min \left\{ 1 , \frac{1}{2M}\right\}$ as follows:
\begin{equation*}
  \bar{T} = \bar{\tau} - \bar{\rho}, \quad \langle\grad \bar{T} , \grad \bar{T}
  \rangle_{\tilde{g}_{\infty}} = -2 + \bar{\rho}^{2} (1 -2M\bar{\rho}) < 0.
\end{equation*}
We then calculate the vector field $\mathcal{F}$ in this context:
\begin{align*}
  &\langle \mathcal{F}_{\tilde{g}_{\infty}} (\bar{T},v), \grad
  \bar{T}\rangle_{\tilde{g}_{\infty}} \\
  &\quad\quad=  \left( 1 - \bar{\rho}^{2}(1-2M \bar{\rho})\right)\left|
    \pd[ \bar{\rho}]v\right|^{2} 
   + \frac{1}{2} \left| \pd[\bar{\tau}]v +
     \bar{\rho}^{2}(1-2M \bar{\rho})\pd[ \bar{\rho}]v\right|^{2} \\
  &\quad\quad\quad
    +  \frac{1}{2}\left|\pd[\bar{\tau}]v\right|^{2}+ \frac{1}{2} \left( 2 -
     \bar{\rho}^{2}(1-2M \bar{\rho})\right)\left( \left| \grad _{\omega}v\right|^{2}
    + v^{2}\right), \\
 & \mathcal{Q}_{\tilde{g}_{\infty}}(\bar{T},v) = - \bar{\rho} (1 - 3M \bar{\rho})
  \left( \left| \pd[ \bar{\rho}]v\right|^{2} + \left| \grad
      _{\omega}\right|^{2} \right).
\end{align*}
The above computation allows us to prove the following Proposition,
whose proof is the same as that of Proposition~\ref{prop:e1} with the
corresponding time-like function:
\begin{proposition}
  \label{prop:i1}
  If $(\phi, \psi)\in C^{\infty}\left( (2M,\infty)_{r} \times
    \sphere^{2}\right)$ satisfy $\supp(\phi)\cup \supp(\psi) \subset
  (2M, \epsilon^{-1})$ for some $\epsilon > 0$, then $\tilde{u}$ is
  smooth up to $\{ \bar{\rho} = 0\}$.
\end{proposition}

We now consider non-compactly supported data.  Near $S_{0}$ in the
partial compactification depicted in Figure~\ref{fig:partial-comp}, we
may use coordinates $(\bar{a},\bar{b})$ given by 
\begin{equation*}
  \bar{a} = -\bar{\rho}\bar{ \tau} = \frac{-t+r+2M\log(r-2M)}{r} \in [0,1], \quad \bar{b} =
  -\frac{1}{\bar{\tau}} .
\end{equation*}
Together with the spherical coordinates $\omega$, these are valid
coordinates near the intersection of $S_{0}$ and $S_{1}^{+}$, where
$\bar{a}$ is a defining function for $E_{1}^{+}$ and $\bar{b}$ is a defining
function for $E_{0}$. Near the interior of $E_{1}^{+}$,  we can take coordinates  $(\bar{\tau},\bar{\rho})$ as before. 
We extend $(\bar{a},\bar{b})$ up to initial surface $t=0$, which defines the smooth structure on this partial compactification near $S_0$. 
Notice that,  at $t=0$, 
$$\bar{a}=\frac{\log(r-2M)}{r}+1.$$
In the coordinate system $(\bar{a},\bar{b},\omega)$, we have that
\begin{align*}
  &\tilde{g}_{\infty} = 2d\bar{a} \left( \frac{\differential{\bar{b}}}{\bar{b}}\right)
  + \bar{a}\left[ 2-\bar{a}(1-2M\bar{a}\bar{b})\right]\left(
    \frac{\differential{\bar{b}}}{\bar{b}}\right)^{2} + \domega^{2}, \\
  &\tilde{\Box}_{\infty} = 2\pd[\bar{a}]\left( \bar{b}\pd[\bar{b}] - \bar{a}\pd[\bar{a}]\right) +
  \left( 1 - 2M\bar{a}\bar{b}\right)(\bar{a}\pd[\bar{a}])^{2} + \left( 1 - 4M\bar{a}\bar{b}\right)\bar{a}\pd[\bar{a}]
  + \lap_{\omega}.
\end{align*}
We also calculate the commutators of $\tilde{\Box}_{\infty}$ with relevant
vector fields:
\begin{align*}
  &\left[ \tilde{\Box}_{\infty} , Z_{ij}\right] = 0,\\
  & \left[ \tilde{\Box}_{\infty}, \bar{b}\pd[\bar{b}]\right] = 2M\bar{a}\bar{b}(\bar{a}\pd[\bar{a}])^{2} +
  4M\bar{a}\bar{b}(\bar{a}\pd[\bar{a}]), \\
  &\left[ \tilde{\Box}_{\infty}, \bar{a}\pd[\bar{a}]\right] =
  \tilde{\Box}_{\infty} - \left( 1 - 4M\bar{a}\bar{b}\right) (\bar{a}\pd[\bar{a}])^{2} -
  (1-8M\bar{a}\bar{b})(\bar{a}\pd[\bar{a}]) - \lap_{\omega}.
\end{align*}
In this coordinate system we now choose the time like functions
$\bar{T}_{1}$ and $\bar{T}_{1}'$ as follows:
\begin{align*}
  &\bar{T}_{1} = -\bar{a} + \log \bar{b}, & & \langle \grad \bar{T}_{1},\grad \bar{T}_{1}\rangle
  _{\tilde{g}_{\infty}} = -2 - \bar{a}\left( 2-\bar{a}(1-2M\bar{a}\bar{b})\right) < 0; \\
  &\bar{T}_{1}' = -\bar{a},& & \langle \grad \bar{T}_{1}', \grad \bar{T}_{1}'
  \rangle_{\tilde{g}_{\infty}} = -\bar{a}\left( 2 - \bar{a} (1-2M\bar{a}\bar{b})\right) < 0.
\end{align*}
The vector field $\mathcal{F}$ and the function $\mathcal{Q}$ are then
given by
\begin{align*}
  &\langle \mathcal{F}_{\tilde{g}_{\infty}} (\bar{T}_{1},v), \grad
  \bar{T}_{1}' \rangle _{\tilde{g}_{\infty}} \\
  &\quad\quad= \frac{1}{2} \left( 2 - \bar{a} (1-2M\bar{a}\bar{b})\right)\bar{a}\left|
    \pd[\bar{a}]v\right|^{2}
  + \frac{1}{2} \left| \bar{a} \left( 2 - \bar{a}
      (1-2M\bar{a}\bar{b})\right)\pd[\bar{a}]v - \bar{b}\pd[\bar{b}]v\right|^{2}  \\
  &\quad \quad\quad
  + \frac{1}{2}\left|
    \bar{b}\pd[\bar{b}]v\right|^{2}+ \frac{1}{2} \left( 1 + \bar{a} \left( 2 - \bar{a}
      (1-2M\bar{a}\bar{b})\right)\right)\left( \left|\grad_{\omega}v\right|^{2} +
    v^{2}\right), \\
  &\mathcal{Q}_{\tilde{g}_{\infty}}(\bar{T}_{1}, v) = \left( 1 - \bar{a}
    (1-3M\bar{a}\bar{b})\right)\left( \left|\pd[\bar{a}]v\right|^{2} - \left|
      \grad_{\omega}v\right|^{2} - v^{2} \right) + M\bar{a}\bar{b}\left|\bar{a} \pd[\bar{a}]v\right|^{2}.
\end{align*}

We are now able to prove the following proposition:
\begin{proposition}
  \label{prop:i2}
  If the initial data $(\phi, \psi)$ lie in the weighted Sobolev space
  \begin{equation*}
     \bar{\rho}^{\lambda+1}H^{N+1}_{b}\left( \left[0,\frac{1}{r_0}\right)_{\bar{\rho}}\times
      \sphere^{2}, \frac{\differential{\bar{\rho}}\domega}{\bar{\rho}}\right) \times
    \bar{\rho}^{\lambda+2}H^{N}_{b}\left( \left[0,\frac{1}{r_0}\right)_{\bar{\rho}}\times \sphere^{2},
      \frac{\differential{\bar{\rho}}\domega}{\bar{\rho}}\right),
  \end{equation*}
  with $r_0+2M\log(r_0-2M)>0$, $N > 2$ and $\lambda > 0$, then $\tilde{u}$ is $C^{\delta}$ up
  to $\{ \bar{a}=0\}$ for $\bar{b} < \bar{b}_{0}$, where $\delta = \min \{ \lambda ,
  \frac{1}{2}\}$ and $\bar{b}_0=1/(r_0+2M\log(r_0-2M))$.
\end{proposition}

\begin{proof}
  The proof is the same as the proof of Proposition~\ref{prop:e2} with
  the corresponding time-like function.  Because $\pd[t]$ lifts to
  $\bar{b}\left( \bar{b}\pd[\bar{b}]-\bar{a}\pd[\bar{a}]\right)$, the assumption on the initial data
  implies that  at the Cauchy surface.
  \begin{equation*}
  Z^I \tilde{u}|_{t=0}  \in \bar{\rho}^{\lambda}H^{N+1-|I|}_{b} \left(
     \left[0,\frac{1}{r_0}\right)_{\bar{\rho}}\times\sphere^{2}, \frac{\differential{\bar{\rho}}\domega}{\bar{\rho}}\right)
  \end{equation*}
for $Z\in\{\bar{a}\partial_{\bar{a}}, \bar{b}\partial_{\bar{b}},Z^{ij}\}$ and multi-index $I$.
\end{proof}

As in the previous section, we have the following corollary:
\begin{corollary}
  \label{cor:i1}
  Suppose $\lambda = k+\alpha$ for some integer $k$ and $\alpha \in
  (0,1]$.  If the initial data $(\phi, \psi) $ lie in the space
  \begin{equation*}
    \bar{\rho}^{\lambda+1}H^{N+1}_{b}\left( \left[0,\frac{1}{r_0}\right)_{\bar{\rho}}\times
      \sphere^{2}, \frac{\differential{\bar{\rho}}\domega}{\bar{\rho}}\right) \times
   \bar{ \rho}^{\lambda+2}H^{N}_{b}\left(  \left[0,\frac{1}{r_0}\right)_{\bar{\rho}}\times \sphere^{2},
      \frac{\differential{\bar{\rho}}\domega}{\bar{\rho}}\right), 
  \end{equation*}
  with $r_0+2M\log(r_0-2M)>0$ and $N>2+k$, then $\tilde{u}$ is $C^{k,\delta}$ up to $\{\bar{a}=0\}$ for $\bar{b} <
  \bar{b}_{0}$, where $\delta = \min \{ \alpha , \frac{1}{2}\}$ and $\bar{b}_0=1/(r_0+2M\log(r_0-2M))$.
\end{corollary}

\begin{proof}
  The proof is the same as the proof of Corollary~\ref{cor:e1} with
  the corresponding time-like function.  This is because
  $\tilde{\Box}_{\infty}$ satisfies:
  \begin{equation*}
    \left[ \tilde{\Box}_{\infty}, \pd[\bar{a}]^{k}\right] = 2k\pd[\bar{a}]^{k+1} +
    \sum _{|I|+i \leq k +1, i\leq k}c_{I,i}Z^{I}\pd[\bar{a}]^{i},
  \end{equation*}
  where $Z\in \{ \bar{a}\pd[\bar{a}], \bar{b}\pd[\bar{b}], Z_{ij}\}$ and $c_{I,i}$ are smooth
  coefficients.  
\end{proof}

Again as before, ``Schwartz'' data behaves in much the same way as
compactly supported data:
\begin{corollary}
  \label{cor:i2}
  For initial data $(\phi, \psi)$ in 
  \begin{equation*}
   \bar{\rho}^{\infty} H^{\infty}_{b}\left( \left[0,\frac{1}{r_0}\right)_{ \bar{\rho}}\times
      \sphere^{2}, \frac{\differential{ \bar{\rho}}\domega}{ \bar{\rho}}\right) \times
   \bar{\rho}^{\infty}H^{\infty}_{b}\left(\left[0,\frac{1}{r_0}\right)_{ \bar{\rho}}\times \sphere^{2} ,
      \frac{\differential{ \bar{\rho}}\domega}{ \bar{\rho}} \right),
  \end{equation*}
  with $r_0+2M\log(r_0-2M)>0$, the rescaled solution $\tilde{u}$ is
  smooth up to $\{\bar{a}=0\}$ for $\bar{b}<
  \bar{b}_{0}=1/(r_0+2M\log(r_0-2M))$.
\end{corollary}

The coordinate system given by $(\bar{a},\bar{b},\omega)$ valid  only on half of null
infinity, i.e. $\bar{\tau}<0$.  To extend Proposition~\ref{prop:i2} to the other half, we
may use similar energy estimates in the $(\bar{\tau}, \bar{\rho}, \omega)$
coordinate system with $\bar{\tau} \in (-\bar{\tau}_{0},\bar{\tau}_{0})$ for arbitrarily
large $\bar{\tau}_{0}$.  

In the coordinate system $(\bar{\tau}, \bar{\rho}, \omega)$, we choose a pair of
time-like functions $(\bar{T}_{2},\bar{T}_{2}')$ for $\bar{\tau} \in (-\bar{\tau}_{0},\bar{\tau}_{0})$
as follows: 
$$\bar{T}_{2}=\bar{T}, \quad \bar{T}_{2}'= -\bar{\rho} (2\bar{\tau}_{0}-\tau),$$
Here $\bar{T}_{2}'$ is asymptotically null when
approaching null infinity 
$$\langle \grad \bar{T}_{2}' , \grad \bar{T}_{2}' \rangle_{\tilde{g}_{\infty}} =
  -\bar{\rho} (2\bar{\tau}_{0} - \bar{\tau}) \left( 2 - (1-2M\bar{\rho}) \bar{\rho}
    (2\bar{\tau}_{0}-\bar{\tau})\right) \leq 0.$$
and  the vector field $\mathcal{F}$ is
\begin{align*}
  & \langle \mathcal{F}_{\tilde{g}_{\infty}}(\bar{T}_{2},v),\grad
  \bar{T}_{2}'\rangle_{\tilde{g}_{\infty}} \\
  &\quad\quad=  \left( 1 - \frac{1}{2}( \bar{\rho} + 2\bar{\tau}_{0} -\bar{\tau})
    (1-2M \bar{\rho}) \bar{\rho}\right)  \bar{\rho}\left| \pd[ \bar{\rho}]v\right|^{2}\\
    &\quad\quad\quad
    + \frac{1}{2} (2\bar{\tau}_{0}-\bar{\tau})
  \left( \left| (1-2M\bar{\rho})\bar{\rho}^{2}\pd[\bar{\rho}]v + \pd[\bar{\tau}]v\right|^{2} +
    \left| \pd[\bar{\tau}]\right|^{2}\right) \\
    &\quad\quad\quad
      + \frac{1}{2} \left( 2\bar{\tau}_{0} - \bar{\tau} + \bar{\rho} -
    (1-2M\bar{\rho})(2\bar{\tau}_{0} - \bar{\tau})\bar{\rho}^{2} \right) \left( \left| \grad
      _{\omega}v\right|^{2} + v^{2}\right).
\end{align*}

We use the pair $(\bar{T}_{2},\bar{T}_{2}')$ of time-like functions to obtain
estimates similar to those in Propositions~\ref{prop:e2} and \ref{prop:i2}:
\begin{proposition}
  \label{prop:i3}
  Suppose $r_{0} > 2M$ is close to $2M$ so that $\bar{\tau} _{0} = -(r_{0} +
  2M \log (r_{0}-2M)) > 0$.  If the initial data is in the following space
  \begin{equation*}
    \bar{\rho}^{\lambda+1}H_{b}^{N+1}\left(\left[0,\frac{1}{r_0}\right)_{\bar{\rho}} \times \sphere^{2},
      \frac{d\bar{\rho}\domega}{\bar{\rho}}\right) \times \bar{\rho}^{\lambda + 2}
    H^{N}_{b} \left( \left[0,\frac{1}{r_0}\right)_{\bar{\rho}}\times \sphere^{2} ,
      \frac{d\bar{\rho}\domega}{\bar{\rho}}\right) 
  \end{equation*}
  with $N > 2$ and $\lambda > 0$, then $\tilde{u}$ is $C^{\delta}$ up
  to $\{\bar{\rho} = 0\}$ for all $\bar{\tau} \in (-\infty, \bar{\tau}_{0})$, where
  $\delta = \min \{ \lambda, \frac{1}{2}\}$.
\end{proposition}

\begin{proof}
  For $\bar{\tau} \leq  -\bar{\tau}_{0}$, i.e.,
  $\bar{b}\leq \bar{b}_{0}$, the statement is proved in Proposition~\ref{prop:i2}.
  Let $\Omega$ be the domain bounded by $\{t=0\}$, $S_1^+$, and
  \begin{equation*}
    S = \{ \bar{T}_{2} = \log \bar{b}_{0} \}=\{\bar{T}_2=-\bar{\tau}_0\}, \quad S' = \{ \bar{T}_{2}= \bar{\tau}_{0}\}.
  \end{equation*}
  Let $\Sigma_{s} = \{ \bar{T}_{2}' = s\}$, which is space-like for $s < 0$
  and approaches null infinity as $s \to 0$.  We define the quantities
  \begin{align*}
    &M_{2}^{N}(\tilde{u},s) = \left(\sum_{|I|\leq N}
      \int_{\Sigma_{s}\cap \Omega}e^{-c\bar{T}_{2}}\langle
      \mathcal{F}_{\tilde{g}_{\infty}}(\bar{T}_{2},Z^{I}\tilde{u}, \grad
      \bar{T}_{2}' \rangle_{\tilde{g}_{\infty}} \dmu_{\bar{T}_{2}'}
    \right)^{\frac{1}{2}}, \\
    &L_{1}^{N}(\tilde{u},s) = \left( \sum_{|I|\leq N} \int _{S\cap
        \{\bar{T}_{2}'<s\}} e^{-c\bar{T}_{2}}\langle
      \mathcal{F}_{\tilde{g}_{\infty}}(\bar{T}_{2},Z^{I}\tilde{u}),\grad
      \bar{T}_{1}\rangle _{\tilde{g}_{\infty}} \dmu_{\bar{T}_{1}}\right)
    ^{\frac{1}{2}}, \\
    &L_{2}^{N}(\tilde{u},s) = \left( \sum_{|I|\leq N} \int_{S'\cap
        \{\bar{T}_{2}' < s\}} e^{-c\bar{T}_{2}}\langle
      \mathcal{F}_{\tilde{g}_{\infty}}(\bar{T}_{2},Z^{I}\tilde{u}), \grad
      \bar{T}_{2}\rangle_{\tilde{g}_{\infty}}\right)^{\frac{1}{2}}, 
  \end{align*}
  where $Z^{I}\in \{ \rho \pd[\rho], \pd[\tau], Z_{ij}\}$.  We choose
  $c$ large enough so that
  \begin{equation*}
    \sum_{|I|\leq N} \operatorname{div}_{\tilde{g}_{\infty}}\left(
      e^{-c\bar{T}_{2}}\mathcal{F}_{\tilde{g}_{\infty}} (\bar{T}_{2},
      Z^{I}\tilde{u}) \right)\leq 0.
  \end{equation*}
  Note that this is possible because $\tilde{u}$ solves the equation
  $\left( \tilde{\Box}_{\infty}+ \gamma_{\infty}\right)\tilde{u} = 0$.  Here
  $c$ depends on $\tau_{0}$ and $N$.
  
  We now choose $s_{0}< 0$ small enough so that $\Sigma_{s_{0}} \cap
  \{ t=0\} \cap \Omega = \emptyset$.  By a proof similar to the one in
  Proposition~\ref{prop:e2}, we have that
  \begin{equation*}
    M_{2}^{N}(\tilde{u},s_{0}) < C_{N},
  \end{equation*}
  where $C_{N}$ is bounded by the initial data norm and depends on
  $s_{0}$.  Moreover, the bound in the proof of
  Proposition~\ref{prop:i2} implies that
  \begin{equation*}
    \pd[s]\left( L_{1}^{N}(\tilde{u},s)\right)^{2} \leq
    \begin{cases}
      C(-s)^{2\lambda-2} & \lambda < \frac{1}{2} \\
      C(-s)^{-\frac{1}{2}} & \lambda \geq \frac{1}{2}
    \end{cases}
  \end{equation*}
  for $s > s_{0}$. 
  Here $C$ depends only on the norm of the initial data, $\bar{\tau}_{0}$,
  and $s_{0}$.

  Consider now the domain $\Omega_{s_{0}}^{s}$ bounded
  by $\Sigma_{s_{0}}$, $\Sigma_{s}$, $S$, and $S'$, for $s> s_{0}$.  By Stokes'
  theorem,
  \begin{align*}
    &\left( M_{2}^{N}(\tilde{u},s)\right)^{2} - \left(
      M_{2}^{N}(\tilde{u},s_{0})\right)^{2} + \left(
      L_{2}^{N}(\tilde{u},s)\right)^{2} - \left(
      L_{2}^{N}(\tilde{u},s_{0})\right)^{2}\\ 
    &\quad\quad \leq \left(
      L_{1}^{N}(\tilde{u},s)\right)^{2} - \left(
      L_{1}^{N}(\tilde{u},s_{0})\right)^{2}. 
  \end{align*}
  Dividing by $s-s_{0}$ and taking a limit then implies that
  \begin{equation*}
    \pd[s]\left( M_{2}^{N}(\tilde{u},s)\right)^{2} \leq \pd[s]\left(
      L_{1}^{N}(\tilde{u},s)\right)^{2} .
  \end{equation*}
  Using the bound above and integrating then shows that, for $s >
  s_{0}$, 
  \begin{equation*}
    M_{2}^{N}(\tilde{u},s) \leq C'\left( 1 + (-s)^{\lambda - \frac{1}{2}}\right).
  \end{equation*}
  An application of Sobolev embedding then shows that $\tilde{u}$ is
  $C^{\delta}$ up to $\{ \bar{\rho} = 0\}$ for $\bar{\tau} \in (-\bar{\tau}_{0},\bar{\tau}_{0})$.
\end{proof}

Similarly, Corollaries~\ref{cor:i1} and \ref{cor:i2} may be extended
to the entire null infinity.
\begin{corollary}
  Suppose $\lambda = k + \alpha$ for an integer $k$ and $\alpha \in
  (0,1]$.  If the initial data $(\phi, \psi)$ is in the space
  \begin{equation*}
     \bar{\rho}^{\lambda+1}H^{N+1}_{b}\left( \left[
        0,\frac{1}{r_{0}}\right)_{\bar{\rho}} \times \sphere^{2} ,
      \frac{d\bar{\rho}\domega}{\bar{\rho}}\right) \times 
     \bar{\rho}^{\lambda
      +2}H^{N}_{b}\left( \left[ 0, \frac{1}{r_{0}}\right)_{\bar{\rho}} \times
      \sphere^{2} , \frac{d\bar{\rho}\domega}{\bar{\rho}}\right)
  \end{equation*}
  with $r_0+2M\log(r_0-2M)<0$ and $N>2+k$, then $\tilde{u}$ is $C^{k,\delta}$ up to $\{\bar{\rho} = 0\}$ for
  $\bar{\tau} < \bar{\tau}_{0} = -(2r_{0} + 2M\log (r_{0}-2M))$.  Here $\delta =
  \min\{\alpha, \frac{1}{2}\}$.
\end{corollary}

\begin{proof}
  For $\bar{\tau} \leq - \bar{\tau}_{0}$, this reduces to Corollary~\ref{cor:i1}.
  We now define:
  \begin{align*}
     &\tilde{M}_{2}^{N}(\tilde{u},s ;\lambda) = \left( \sum_{i\leq k}
      \left(
        M_{2}^{N-1}(\pd[\bar{\rho}]^{i}\tilde{u},s;\lambda)\right)^{2}\right)^{\frac{1}{2}},
    \\
     &\tilde{L}_{1}^{N}(\tilde{u},s; \lambda) = \left( \sum_{i\leq k}
      \left(
        L_{1}^{N-i}(\pd[\bar{\rho}]^{i}\tilde{u},s;\lambda)\right)^{2}\right)^{\frac{1}{2}},
    \\ 
     &\tilde{L}_{2}^{N}(\tilde{u},s;\lambda) = \left( \sum_{i\leq k}
      \left( L_{2}^{N-i}(\pd[\bar{\rho}]^{i}\tilde{u},s;\lambda)\right)^{2}\right)^{\frac{1}{2}}.
  \end{align*}
  Notice that
  \begin{equation*}
    \left[ \tilde{\Box}_{\infty} , \pd[\bar{\rho}]^{k} \right] =
    \sum_{|I|+i\leq k+1,\,i\leq k} c_{I,i}Z^{I}\pd[\bar{\rho}]^{i},
  \end{equation*}
  where $Z\in \{a\pd[a],b\pd[b],Z_{ij}\}$ and $c_{I,i}$ are smooth
  coefficients.  By a proof similar to that of Corollary~\ref{cor:e1},
  we have:
  \begin{align*}
    &\pd[s]\left( \tilde{M}^{N}_{2}(\tilde{u},s;\lambda)\right)^{2} +
       \pd[s]\left( \tilde{L}_{2}^{N}(\tilde{u},s;\lambda)\right)^{2}\\
    &\quad\quad \leq \pd[s]\left(
      \tilde{L}_{1}^{N}(\tilde{u},s;\lambda)\right)^{2}
    \leq
    \begin{cases}
      C\left( -s \right)^{2\alpha -2}, & \alpha < \frac{1}{2} \\
      C\left( -s \right)^{-\frac{1}{2}}, & \alpha \geq \frac{1}{2}
    \end{cases}.
  \end{align*}
  We thus have that $\pd[\bar{\rho}]^{k}\tilde{u}$ is $C^{\delta}$ up to $\{\bar{\rho}
  = 0\}$ for $\bar{\tau} < \bar{\tau}_{0}$.  
\end{proof}

\begin{corollary}
  \label{cor:i4}
  If the initial data are ``Schwartz'', i.e., $(\phi, \psi)$ lie in
  \begin{equation*}
      \bar{\rho}^{\infty}H^{\infty}_{b}\left( \left[0,
        \frac{1}{r_{0}}\right)_{\bar{\rho}}\times \sphere^{2} ,
      \frac{d\bar{\rho}\domega}{\bar{\rho}}\right) \times 
      \bar{\rho}^{\infty}H^{\infty}_{b} \left( \left[ 0, \frac{1}{r_{0}}\right)_{\bar{\rho}} \times
      \sphere^{2}, \frac{d\bar{\rho}\domega}{\bar{\rho}}\right),
  \end{equation*}
 with $r_0+2M\log(r_0-2M)<0$, then $\tilde{u}$ is smooth up to $\{ \bar{\rho}= 0\}$ for $\bar{\tau}< \bar{\tau}_{0} = -
  \left( 2r_{0} + 2M\log(r_{0}-2M)\right)$.
\end{corollary}

Suppose now that the initial data has a classical asymptotic expansion
at $S_{0}$ (spatial infinity), i.e.,
\begin{equation*}
  (\phi, \psi) \in \bar{\rho}^{\lambda+1}C^{\infty}\left( \left[ 0,
      \frac{1}{r_{0}}\right)_{\bar{\rho}}\times \sphere^{2}\right) \times
  \bar{\rho}^{\lambda+2}C^{\infty}\left( \left[ 0 , \frac{1}{r_{0}}\right)_{\bar{\rho}}
    \times \sphere^{2}\right).
\end{equation*}
with $r_0+2M\log(r_0-2M)<0$. Define the ``k-th term'' $w^{k}$ as follows:
\begin{equation}
\label{kthterm.nullinfinity}
 w^{0} = \tilde{u} =\bar{\rho}^{-1}u, \quad
 w^{k} = \prod_{i=0}^{k-1}\left(
  \bar{b}\partial_{\bar{b}}- \lambda -i\right)^{n_i}\tilde{u}\quad \mathrm{for}\quad k\geq 1,
\end{equation} 
where $n_i$ are integers that only depend on $i$ such that
$$(Z^Iw^k)|_{t=0}\in \sum_{j=1}^{N(k,I)}\bar{\rho}^{\lambda+k}(\log\bar{\rho})^jC^{\infty}\left( \left[ 0,
      \frac{1}{r_{0}}\right)_{\bar{\rho}}\times \sphere^{2}\right)$$
for $Z\in\{\bar{a}\pd[\bar{a}], \bar{b}\pd[\bar{b}], Z^{ij}\}$, multi-index $I$ and integers $N(k,I)$. 
Here the integers $n_i$ and $N(k, I)$ can be calculated according to the lift of 
$\bar{a}\pd[\bar{a}], \bar{b}\pd[\bar{b}]$:
\begin{align*}
&\bar{a}\pd[\bar{a}]=\bar{\rho}\partial_{\bar{\rho}}-\frac{1}{\bar{\rho}(1-2M\bar{\rho})}\pd[t],
\\
&\bar{b}\partial_{\bar{b}}=\bar{\rho}\partial_{\bar{\rho}}-2M\left(\log\bar{\rho}+\frac{1}{1-2M\bar{\rho}}
-\log(1-2M\bar{\rho})\right)\partial_t-t\partial_t.
\end{align*}

\begin{proposition}
  \label{prop:i4}
  If $\lambda + k = l + \alpha$ for some integer $l$ and $\alpha \in
  (0,1]$, then $w^{k}$ is $C^{l,\delta}$ up to $\{ \bar{\rho} = 0\}$ for all
  $\bar{\tau} < \bar{\tau}_{0} = - \left( 2r_{0} + 2M\log(r_{0}-2M)\right)$, where
  $\delta = \min\{\frac{1}{2}, \alpha - \epsilon\}$ with $\epsilon > 0$ arbitrarily small.
\end{proposition}

\begin{proof}
  We apply a similar proof to that of Proposition~\ref{prop:e3} for
  $\bar{\tau} < -\bar{\tau}_{0}$, i.e. for $\bar{b} < \bar{\tau}_{0}^{-1}$, and then extend
  it to $\bar{\tau} < \bar{\tau}_{0}$ by a method similar to the method in
  Proposition~\ref{prop:i3}.  
\end{proof}

\section{The radiation field}
\label{sec:extens-energy-space}

In this section we define the radiation field, show it extends to an
energy space, and then show it is unitary (i.e., norm-preserving) on
this space.

Consider now smooth initial data $(\phi , \psi)$ supported in $r \geq
2M+\epsilon$ and $r \leq R_{0}$, and let $u$ be the solution of $\Box_{S}
u =0$ with this initial data.  By the results of the last section, $u$
restricts to the event horizon and $ ru$ restricts to null
infinity.  
\begin{definition}
  The \emph{forward radiation field of $(\phi, \psi)$ at the event
    horizon} is the restriction of $\pd[t]u$ to the event horizon.  In
  the coordinates $(\tau, \rho, \omega)$ of
  Section~\ref{sec:event-horizon}, it is given by
  \begin{equation*}
    \mathcal{R}_{\horiz}(\phi, \psi) (\tau, \omega) =
    \pd[t]u(\tau, \rho, \omega) |_{\rho = 0}.
  \end{equation*}
  
  The \emph{forward radiation field of $(\phi, \psi)$ at null
    infinity} is the restriction of $\pd[t]v$ to null infinity.  In
  the coordinates $(\bar{\tau}, \bar{\rho}, \omega)$ of
  Section~\ref{sec:near-future-null}, it is given by
  \begin{equation*}
    \mathcal{R}_{\scri}(\phi, \psi) (\bar{\tau}, \omega) =
    \bar{\rho}^{-1}\pd[t]u (\bar{\tau}, \bar{\rho}, \omega)|_{\bar{\rho} = 0}.
  \end{equation*}
\end{definition}
Finite speed of propagation implies that
$\mathcal{R}_{\horiz}(\phi,\psi) (\tau, \omega)$ vanishes identically
for $\tau \leq 2M + \epsilon + 2M\log \epsilon$ and that
$\mathcal{R}_{\scri}(\phi,\psi)(\bar{\tau},\omega)$ vanishes
identically for $\bar{\tau}\leq - \left( R_{0} + 2M \log ( R_{0} -
  2M)\right)$.

The existence of the static Killing field $\pd[t]$ implies that the energy
\begin{equation*}
  E(t) = \frac{1}{2}\int_{\sphere^{2}}\int_{2M}^{\infty} e(t)r^{2}\dr\domega
\end{equation*}
is conserved, where
\begin{equation*}
e(t)=  \left( 1 -
      \frac{2M}{r}\right)^{-1}\left( \pd[t]u\right)^{2} + \left( 1 -
      \frac{2M}{r}\right) \left( \pd[r]u\right)^{2} +
    \frac{1}{r^{2}}\left| \grad_{\omega}u\right|^{2}.
\end{equation*}
For a fixed $\lambda$, let us write
\begin{align*}
  E(t) = &\frac{1}{2}\int_{\sphere^{2}}\int_{-r-2M\log (r-2M) \geq t - \lambda} e(t) r^{2}\dr\domega \\
  	&+ \frac{1}{2}\int_{\sphere^{2}}\int_{\lambda-t\leq r + 2M\log (r-2M) \leq t-\lambda} e(t) r^{2}\dr\domega \\
  &+ \frac{1}{2}\int_{\sphere^{2}}\int_{r + 2M\log (r-2M) \geq t -\lambda} e(t) r^{2}\dr\domega \\
  =& \text{I} + \text{II} + \text{III}.
\end{align*}

Let $f_{\horiz} = \mathcal{R}_{\horiz}(\phi, \psi)(\tau,\omega)$ and
$f_{\scri} = \mathcal{R}_{\scri}(\phi,\psi)(\bar{\tau},\omega)$.

We now compute term $\text{I}$ in the $(\tau,\rho, \omega)$ coordinates
of Section~\ref{sec:event-horizon}. Recall that near event horizon, 
$$\tau=t+r+2M\log(r-2M), \quad\rho=r-2M. $$ 
So $\pd[t]$ lifts to $\pd[\tau]$
and $\pd[r]$ lifts to $\pd[\rho] + \left( \rho+2M\right)\rho^{-1}\pd[\tau]$.  Moreover, $t$ is fixed, so we
are free to use $\tau$ as the variable of integration (so we think of
$r$ as a function of $\tau$) and then $r^2drd\omega=(\rho+2M)\rho d\tau d\omega$.  We then have that
\begin{align*}
  \text{I} &= \frac{1}{2}\int_{\sphere^{2}}\int _{\tau \leq \lambda}\left( 2 \left|\pd[\tau]u\right|^{2} 
  + \frac{2\rho\pd[\rho]u\pd[\tau]u}{(\rho+2M)} +
     \frac{\left| \rho\pd[\rho]u\right|^{2}}{(\rho+2M)^2}\right.\\
 &\quad\quad\quad\quad\quad\quad  \left.+
    \frac{\rho\left| \grad_{\omega}u\right|^{2}}{(\rho+2M)^{3}}\right)(\rho+2M)^{2}\dtau\domega.
\end{align*}
Because $\mathcal{R}_{\horiz}(\phi, \psi)$ vanishes for $\tau \leq
\tau _{0}$, the convergence of $\pd[\tau]u$ (together with its derivatives) to
$\mathcal{R}_{\horiz}(\phi,\psi)$ is uniform on $(-\infty, \lambda ]$, and so we may take a
limit as $t\to \infty$.  This implies that
\begin{equation*}
  \lim _{t\to \infty}\text{I} = 4M^{2}\int_{\sphere^{2}}\int
  _{-\infty}^{\lambda}f_{\horiz}^{2} \dtau\domega .
\end{equation*}

We then compute term $\text{III}$ in the $(\bar{\tau},\bar{\rho},\omega)$ 
coordinates of Section~\ref{sec:near-future-null}. Near future null infinity, 
$$\bar{\tau}=t-r-2M\log(r-2M),\quad \bar{\rho}=\frac{1}{r}.$$ 
So $\pd[t]$ lifts to $\pd[\tau]$ and $\pd[r]$
lifts to $- \bar{\rho}^{2}\pd[\bar{\rho}] - \left( 1 - 2M\bar{\rho}\right)^{-1}\pd[\bar{\tau}]$, and
the volume form $r^{2}\dr\domega$ becomes
$\bar{\rho}^{-4}d\bar{\rho}\domega$.  We are again free to use $\bar{\tau}$ as the
variable of integration, so that $d\bar{\tau} = \left( 1 - 2M\bar{\rho}\right)^{-1}
\bar{\rho}^{-2}d\bar{\rho}$.  
We then have that
\begin{align*}
  \text{III} &= \frac{1}{2} \int_{\sphere^{2}}\int _{\bar{\tau}\leq \lambda}
  \bigg( 2 \left|\bar{\rho}^{-1} \pd[\bar{\tau}] u\right|^{2}  +
 2(1-2M\bar{\rho})\pd[\bar{\rho}]u\pd[\bar{\tau}]u+(1-2M\bar{\rho})^2|\bar{\rho}\pd[\bar{\rho}]u|^2 
     \\ 
    &\quad\quad\quad\quad\quad\quad
        +(1-2M\bar{\rho})|\nabla_{\omega}u|^2 \bigg) d\bar{\tau}\domega .
\end{align*}
We now use the vanishing of $\mathcal{R}_{\scri}(\phi, \psi)$ for $\bar{\tau}
\leq \bar{\tau}_{0}$ to conclude that the convergence of $\bar{\rho}^{-1}\pd[\bar{\tau}]u$ (and its
derivatives) to $\mathcal{R}_{\scri}(\phi,\psi)$ 
is uniform on $(-\infty, \lambda]$.  This allows us to take a limit as
$t\to \infty$ and conclude that
\begin{equation*}
  \lim_{t\to\infty}\text{III} =
  \int_{\sphere^{2}}\int_{-\infty}^{\lambda}f_{S_1^+}^{2}d\bar{\tau}\domega .
\end{equation*}

Term $\text{II}$ is positive, so we may now take a limit as $\lambda
\to \infty$ to conclude that
\begin{align}
  \label{eq:energy-identity-unitary}
  E(0) &= 4M^{2}\norm[L^{2}(\reals\times\sphere^{2})]{f_{\horiz}}^{2}
  +
  \norm[L^{2}(\reals\times\sphere^{2})]{f_{\scri}}^{2} \\
  &\quad +\lim _{\lambda \to \infty}\lim_{t\to\infty}
  \frac{1}{2}\int_{\sphere^{2}}\int_{ \lambda-t\leq r + 2M\log (r-2M) \leq t-\lambda} 
       e(t)r^{2}\dr\domega .\notag
\end{align}

We define the energy space $\energyspace$ of initial data to be the
completion of smooth functions compactly supported in
$(2M,\infty)\times \sphere^{2}$ with respect to the energy norm
\begin{align*}
  \norm[\energyspace]{(\phi,\psi)}^{2} = &\frac{1}{2}\int_{\sphere^{2}}\int
  _{2M}^{\infty}\left[ \left( 1 -
      \frac{2M}{r}\right)^{-1}\psi^{2} + \left( 1 -
      \frac{2M}{r}\right) \left( \pd[r]\phi\right)^{2} \right.\\
    &\quad\quad\quad\quad \left. + \frac{1}{r^{2}}\left|
      \grad_{\omega}\phi\right|^{2}\right]r^{2}\dr\domega. 
\end{align*}
The above computation then shows that the radiation field extends to
a bounded operator $\energyspace\to L^{2}(\reals\times
\sphere^{2})\oplus L^{2}(\reals \times \sphere^{2})$.

We now show that $\mathcal{R}_{+}$ is unitary (but not necessarily
surjective), i.e., that the last term in
equation~\eqref{eq:energy-identity-unitary} vanishes.  We first show
this for compactly supported smooth initial data and then extend by
density to the full energy space.  Versions of this statement were
shown by Dimock~\cite{Dimock:1985} and Bachelot~\cite{Bachelot:1994},
though they work in a more abstract scattering framework.

The main tool we use to establish unitarity is the formulation of
Price's law proved by Tataru~\cite{Tataru:2013}, which roughly states
that waves decay at the order of $\langle t \rangle^{-3}$ on the
Schwarzschild exterior.  This formulation is somewhat stronger than
the standard formulation of Price's law near the event horizon.  In
particular, we refer to the formulation of
Metcalfe--Tataru--Tohaneanu~\cite{Metcalfe:2012}.  We do not need the
full strength of Price's law; a decay bound on the energy in all
forward cones of smaller aperture than a light cone should suffice.
The local energy decay bounds in the literature, however, do not seem
quite strong enough for our purposes.

In the ``normalized coordinates'' of Tataru~\cite{Tataru:2013} and
Metcalfe--Tataru--Tohaneanu~\cite{Metcalfe:2012}, one has that solutions of the
wave equation with smooth, compactly supported data satisfy the
following pointwise bounds:
\begin{equation*}
  \left| u \right| \lesssim \frac{1}{\langle t^* \rangle \langle
    t^*-r^* \rangle^{2}}, \quad \left| \pd[t]u \right| \lesssim
  \frac{1}{\langle t^* \rangle \langle t^*-r^* \rangle ^{3}}.
\end{equation*}
These coordinates correspond to modified Regge--Wheeler coordinates
and therefore in our setting, 
\begin{align*}
&t^*\sim t,\ r^*\sim r+2M\log(r-2M), &\quad\textrm{near null infinity},\\
&t^*\sim \tau,\ r^*\sim r, &\quad\textrm{near the event horizon}.
\end{align*}
In particular, in our coordinates $u$ is bounded\footnote{One should
  think of the $\langle t \rangle^{-1}$ as governing the decay we
  factor out when defining the radiation field.} by $\langle t\rangle
^{-1} \langle \bar{\tau}\rangle^{-2}$ near null infinity and by
$\langle \tau\rangle ^{-1} \langle t + 2M \log (r-2M) \rangle ^{-2}$
near the event horizon.  

We now use the pointwise bounds to bound the integral $\text{II}$
above.  Let us implicitly insert cutoff functions and treat it as two
integrals - one for large $r$ and another for small $r$.  For large
$r$, inserting the pointwise bounds yields an integral bounded by
\begin{equation*}
  \int _{r_{0}}^{R^{-1}(t-\lambda)} \int_{\sphere^{2}}\frac{1}{\langle
    t\rangle^{2}\langle \bar{\tau}\rangle^{4}} r^{2}\dr \domega \lesssim
  \frac{1}{\langle t \rangle ^{2} \lambda} + \frac{2}{\langle t\rangle
    \lambda ^{2}} + \frac{1}{\lambda^{3}}.
\end{equation*}
For $r$ near $2M$, the pointwise bounds yield an integral bounded by
\begin{align*}
  &\int _{ R^{-1}(\lambda-t)}^{r_0}\int
  _{\sphere^{2}}\frac{1}{\langle \tau\rangle^{2}
    \langle t + 2M \log (r-2M)\rangle ^{4}}
  r^{2} \dr \domega \\
  &\quad\quad  \lesssim \frac{C_{r_{0}}}{\langle t + (\lambda - t - 2M)\rangle^{6}}
  \lesssim \lambda ^{-6}.
\end{align*}
Here $R^{-1}$ is the inverse function of $R:(2M,\infty)\longrightarrow (-\infty,\infty)$ defined by 
$R(r)=r+2M\log(r-2M)$. 
By sending $t \to \infty$ and then sending $\lambda \to
\infty$, we see that the additional contribution of term $\text{II}$
tends to $0$.  This yields that for smooth, compactly supported data,
we have
\begin{equation*}
  E(0) = 4M^{2}\norm[L^{2}(\reals \times \sphere^{2})]{f_{\horiz}}^{2} +
  \norm[L^{2}(\reals \times \sphere^{2})]{f_{\scri}}^{2},
\end{equation*}
i.e., the radiation field is unitary on this subpace.  A standard
density argument then shows that the radiation field is a unitary
operator on the space of initial data with finite energy norm.  This
finishes the proof of Theorem~\ref{thm:unitary1}.  

\section{Support theorems for the radiation field}
\label{sec:supp-theor-radi}

In this section we prove two support theorems for smooth compactly
supported initial data using methods of S{\'a}
Barreto~\cite{Sa-Barreto:2003} and the first author and S{\'a}
Barreto~\cite{BaskinBarreto2012}.  We first show that for such data,
the support of the radiation field on the event horizons provides a
restriction on the support of the data.  We then show a similar
theorem for null infinity.  In particular, for initial data
$(0,\psi)$, $\psi\in C^{\infty}_{c}\left( (2M, \infty) \times
  \sphere^{2}\right)$, we show that if the radiation field vanishes on
$E_{1}^{+}$, then $\psi \equiv 0$.  In other words, it is impossible
to find initial data in this class so that the radiation field is
supported only at null infinity.  We also show the corresponding
statements for null infinity ($S_{1}^{+}$).

First, let us define the two components of 
the backward radiation field as follows:
\begin{align}
  \label{eq:backward-rad-field}
 &\mathcal{R}_{\scrim}(\phi,\psi)(\tau,\omega)=\lim_{r\rightarrow 2M} \pd[t]u(\tau+r+2M\log(r-2M),r,\omega),\\
 &\mathcal{R}_{\horizm}(\phi,\psi)(\bar{\tau},\omega)=\lim_{r\rightarrow
   \infty} r\pd[t]u(\bar{\tau}-r-2M\log(r-2M),r,\omega). \notag
\end{align}

\subsection{At the event horizon}
\label{sec:at-event-horizon}

The main result of this section is the following theorem:
\begin{theorem}
  \label{thm:support-horiz}
  Suppose $\phi, \psi \in C^{\infty}_{c}\left( (2M, \infty) \times
    \sphere^{2}\right)$.  If $\mathcal{R}_{\horiz}(\phi, \psi)$
  vanishes for $\tau \leq \tau_{0}$ and
  $\mathcal{R}_{\horizm}(\phi, \psi)$ vanishes for $\tau \geq
  -\tau_{0}$, then both $\phi$ and $\psi$ are supported in $[r_{0},
  \infty)\times \sphere^{2}$, where $r_{0}$ is given implicitly by
  \begin{equation*}
    r_{0} + 2M\log (r_{0}-2M) = \tau_{0}.
  \end{equation*}
\end{theorem}

\begin{corollary}
  \label{cor:odd-data-horizon}
  Suppose $\psi \in C^{\infty}_{c}\left(
    (2M,\infty)\times\sphere^{2}\right)$.  If
  $\mathcal{R}_{\horiz}(0,\psi)$ vanishes for $\tau \leq \tau_{0}$,
  then $\psi$ is supported in $[r_{0}, \infty)\times\sphere^{2}$.  In
  particular, if $\mathcal{R}_{\horiz}(0,\psi)$ vanishes
  identically, then $\psi \equiv 0$.  
\end{corollary}

\begin{proof}[Proof of Corollary~\ref{cor:odd-data-horizon}]
  By the time-reversibility of the wave equation, we know that
  \[
  \mathcal{R}_{\horiz}(0,\psi)
  (\tau,\omega) = \mathcal{R}_{\horizm}(0,\psi)(-\tau,\omega),
  \]
  and so we may apply Theorem~\ref{thm:support-horiz}.
\end{proof}

We showed in Section~\ref{sec:event-horizon} that solutions with
compactly supported smooth initial data are smooth up to the event
horizons.  In fact, the argument above shows they are jointly smooth
at the event horizons.  Here $\mu$ and $\nu$ are Kruskal coordinates,
given in terms of $\tau_{+} = t + r + 2M\log(r-2M)$ and $\tau_{-} = t
- r - 2M \log (r-2M)$ as follows:
\begin{equation*}
  \mu = e^{\frac{\tau_{+}}{4M}}, \ \nu= e^{-\frac{\tau_{-}}{4M}}.
\end{equation*}
Note that near $r=2M$ (and therefore in the entire exterior domain),
$r$ is a smooth function of $\mu$ and $\nu$.  In these coordinates,
the Schwarzschild metric has the following form:
\begin{equation*}
  g_{S} =  \frac{16M^{2}e^{-r/2M}}{r} \dmu\dnu + r^{2}\domega^{2}.
\end{equation*}
Observe that $\mu$ is a defining function for the past event horizon
$E_{1}^{-}$ while $\nu$ is a defining function for the future event
horizon $E_{1}^{+}$.

We summarize the joint smoothness in the following lemma:
\begin{lemma}
  \label{lem:smooth-horiz}
  If $\phi, \psi \in C^{\infty}_{c}\left(
    (2M,\infty)\times\sphere^{2}\right)$ and $u$ solves the initial
  value problem~\eqref{eq:IVP} with data $(\phi, \psi)$, then $u$ is
  smooth as a function of $\mu$ and $\nu$.
\end{lemma}

\begin{proof}[Proof of Theorem~\ref{thm:support-horiz}]
  Let $c_{0} = e^{\tau_{0}/4M}$.  We start by showing that $u$ must
  vanish to infinite order at $\nu = 0$ for $\mu \in [0, c_{0}]$ and
  at $\mu = 0$ for $\nu \in [0,c_{0}]$.  Indeed, consider the
  D'Alembertian $\Box_{S}$, which is given in these coordinates by
  \begin{equation*}
    \Box_{S} = \frac{e^{r}}{8M^{2}r}\pd[\mu]\left( r^{2}
      \pd[\nu]\right) + \frac{e^{r}}{8M^{2}r}\pd[\nu]\left( r^{2}
      \pd[\mu]\right) + \frac{1}{r^{2}}\lap_{\omega},
  \end{equation*}
  where $r$ is regarded as a smooth function of $\mu$ and $\nu$.
  Observe that near $\mu=0$ or $\nu = 0$, $r-2M \sim \mu \nu$.

  Because the initial data is compactly supported and smooth, there is
  some constant $a_{0}$ so that $u(\mu, \nu , \omega)$ vanishes for
  $(\mu, \nu) \in [0,a_{0}]^{2}$.  Moreover, the assumption on the
  radiation field implies that $\pd[\mu]u$ vanishes when $\nu=0$ and
  $\mu \in [0,c_{0}]$, while $\pd[\nu]u$ vanishes for $\mu=0$ and $\nu
  \in [0,c_{0}]$.  Moreover, because the initial data is smooth and
  compactly supported, one may write (with a similar expression for
  $u(\mu, 0,\omega)$
  \begin{equation*}
    u (0,\nu,\omega) = \int _{0}^{\nu}\pd[\nu]u(0,s,\omega)\ds,
  \end{equation*}
  to conclude that in fact $u$ vanishes for $\{\mu = 0, \nu \in
  [0,c_{0}]\}$ and $\{\nu = 0, \mu \in [0,c_{0}]\}$.  

  We now work exclusively at $\nu=0$ (as the $\mu=0$ case is handled
  in the same way).  Because $u$ satisfies $\Box_{S} u = 0$, we have
  that for $\nu = 0$ and $\mu\in [0,c_{0}]$,
  \begin{equation*}
    \frac{2e^{r}}{8M^{2}r}\pd[\nu]\pd[\mu] u +
    \frac{e^{r}}{8M^{2}r}\pd[\mu](r^{2})\pd[\nu]u +
    \frac{1}{r^{2}}\lap_{\omega}u  = 0.
  \end{equation*}
  Because $u(\mu,0,\omega) = 0$ for $\mu\in [0,c_{0}]$,
  $\frac{1}{r^{2}}\lap_{\omega}u = 0$ here as well.  Moreover, we note
  that $\pd[\mu]r^{2} \sim \nu$, so the second term vanishes at
  $\nu=0$ as well.  (Recall that we already know $u$ is smooth as a
  function of $(\mu,\nu,\omega)$.)  In particular, $\pd[\nu]\pd[\mu]u
  (\mu, 0, \omega) = 0$.  By integrating in $\mu$, we may also
  conclude that $\pd[\nu]u(\mu,0,\omega) = 0$ for $\mu \in
  [0,c_{0}]$.  

  Differentiating the equation $\frac{8M^{2}r}{e^{r}}\Box_{S}u=0$ in $\nu$ yields
  \begin{equation*}
    2\pd[\nu]^{2}\pd[\mu]u + \pd[\mu](r^{2}) \pd[\nu]^{2}u +
    \pd[\nu]\pd[\mu](r^{2})\pd[\nu]u +
    \frac{8M^{2}}{r^{2}e^{r}}\lap_{\omega}\pd[\nu]u + \pd[\nu]\left(
      \frac{8M^{2}}{re^{r}}\right)\lap_{\omega}u=0.  
  \end{equation*}
  Using that $\pd[\nu]u$, and $u$ vanish at $(\mu,0,\omega)$ and that
  $\pd[\mu](r^{2})$ vanishes at $\nu = 0$, we conclude that
  $\pd[\nu]^{2}\pd[\mu]u(\mu,0,\omega) = 0$.  Integrating again, we
  also conclude that $\pd[\nu]^{2}u(\mu,0,\omega)=0$.  

  Continuing inductively, after differentiating the equation $k$
  times, we may conclude that $\pd[\nu]^{k+1}u$ vanishes at
  $(\mu,0,\omega)$ for $\mu \in [0,c_{0}]$.  In particular, $u$
  vanishes to infinite order at $\nu=0$ for $\mu \in [0,c_{0}]$.  

  We now decompose $u$ into spherical harmonics.  We set $w_{j}$ to be
  an eigenbasis of $L^{2}(\sphere^{2})$ with eigenvalues
  $-\lambda_{j}^{2}$ and write $u(t,r,\omega) =
  \sum_{j=0}^{\infty}u_{j}(t,r) w_{j}(\omega)$.  Note that this
  decomposition extends to the radiation field as well.  The
  smoothness of the initial data (and thus the solution) implies that
  the terms in the series are rapidly decreasing in $j$ and so the
  series converges uniformly.  In particular, each $u_{j}$ vanishes to
  infinite order at $\nu=0$ for $\mu\in [0,c_{0}]$ (with a
  corresponding statement at the other horizon).  In terms of the
  coordinates $\mu$ and $\nu$, each $u_{j}$ solves the equation
  \begin{equation}
    \label{eq:eq-for-uj-horiz}
    \pd[\nu]\left( r^{2}\pd[\mu]u_{j}\right) + \pd[\mu]\left( r^{2}
      \pd[\nu]u_{j}\right) - \frac{8M^{2}\lambda_{j}^{2}}{re^{r}}u_{j} =0.
  \end{equation}
  In particular, we may extend $u_{j}$ by $0$ to $\nu < 0$, $\mu\in
  [0,c_{0}]$ and $\mu < 0$, $\nu \in [0,c_{0}]$.  In other words, we
  may extend $u_{j}$ to a smooth function for $\mu,\nu\in (-\infty, c_{0})$ that
  vanishes to infinite order at $\mu + \nu = 0$ for $\mu \in
  [-c_{0},c_{0}]$.  
  
  We now use a unique continuation argument for the $1+1$-dimensional
  wave equation to conclude that each $u_{j}$ must vanish in the
  region we claim.  Indeed, because the extension of $u_{j}$ vanishes
  to infinite order at $\mu + \nu = 0$ for $\mu \in [-c_{0},c_{0}]$,
  finite speed of propagation (with respect to $\mu + \nu$) implies
  that $u$ vanishes identically for $\mu$ and $\nu$ in the domain of
  dependence of $\{ \mu + \nu = 0, \mu \in [-c_{0},c_{0}]\}$, i.e., in
  the region $\{ \mu \leq c_{0}, \nu \leq c_{0}\}$.  In particular,
  the solution vanishes identically for $\{ \mu = \nu, \mu, \nu \leq
  c_{0}\}$.  Because the hypersurface $\{\mu = \nu\}$ agrees with the
  hypersurface $\{ t=0\}$ away from $\mu = \nu = 0$, we may conclude
  that the spherical components of the initial data are supported in
  $\mu \geq c_{0}$, i.e., in $r \geq r_{0}$.
\end{proof}

\subsection{At null infinity}
\label{sec:at-null-infinity}

The main result of this section is the corresponding theorem at null infinity:
\begin{theorem}
  \label{thm:support-scri}
  Suppose $\phi, \psi \in C^{\infty}_{c}\left( (2M, \infty) \times
    \sphere^{2}\right)$.  If $\mathcal{R}_{\scri}(\phi,\psi)$
  vanishes for $\bar{\tau} \leq -\bar{\tau}_{0}$ and
  $\mathcal{R}_{\scrim}(\phi,\psi)$ vanishes for $\bar{\tau} \geq
  \bar{\tau}_{0}$, then both $\phi$ and $\psi$ are supported in $(2M,
  r_{0}] \times \sphere^{2}$, where $r_{0}$ is given implicitly by
  \begin{equation*}
    r_{0} + 2M\log (r_{0}-2M) = \bar{\tau}_{0}.
  \end{equation*}
\end{theorem}

\begin{corollary}
  \label{cor:odd-data-scri}
  Suppose $\psi \in C^{\infty}_{c}\left( (2M, \infty)\times
    \sphere^{2}\right)$.  If $\mathcal{R}_{\scri}(0,\psi)$ vanishes
  for $\bar{\tau} \leq - \bar{\tau}_{0}$, then $\psi$ is supported in
  $(2M,r_{0}]\times \sphere^{2}$.  In particular, if
  $\mathcal{R}_{\scri}(0,\psi)$ vanishes identically, then $\psi
  \equiv 0$.  
\end{corollary}

\begin{proof}[Proof of Corollary~\ref{cor:odd-data-scri}]
  Again by the time-reversibility of the wave equation, one has that
  \[
  \mathcal{R}_{\scrim}(0,\psi)(\bar{\tau}, \omega) =
  \mathcal{R}_{\scri}(0,\psi)(-\bar{\tau},\omega),
  \]
  and we may apply Theorem~\ref{thm:support-scri}.
\end{proof}

\begin{proof}[Proof of Theorem~\ref{thm:support-scri}]
  We showed in Section~\ref{sec:near-future-null} that solutions with
  compactly supported smooth initial data have expansions at null
  infinity.  If the initial data are in $C_{c}^{\infty}\left( (2M,
    \infty)\times\sphere^{2}\right)$, then the rescaled solution
  $ru$ is smooth as a function of $\bar{\tau}$
  up to $\bar{\rho} = 0$. Here $\bar{\rho}=1/r$ and we use $\bar{\tau}_+,\bar{\tau}_-$ to 
  distinguish the coordinates at future and past null infinity if necessary,
  i.e., 
  $$\bar{\tau}_+=t-r-2M\log(r-2M), \quad\bar{\tau}_-=t+r+2M\log(r-2M).$$ 

  We start by showing that if the initial data $(\phi, \psi)$ are
  smooth and compactly supported, and $\mathcal{R}_{\scri}(\phi,
  \psi)$ vanishes for $\bar{\tau}_+ \leq -\bar{\tau}_{0}$, then $v = r u$
  vanishes to infinite order there. A similar argument applies near
  past null infinity as well.  First note that by integrating in
  $\bar{\tau}_+$, we see that $v$ vanishes at $\bar{\rho} = 0$ for $\bar{\tau}_+ \leq -
  \bar{\tau}_{0}$.  We now use that $v$ satisfies the following equation
  \begin{equation*}
    \tilde{\Box}_{\infty}v + \gamma_{\infty} v = 0,
  \end{equation*}
  where $\gamma_{\infty}$ is a smooth function across $\bar{\rho} = 0$ and
  $\tilde{\Box}_{\infty}$ is the D'Alembertian for the conformal metric
  $\bar{\rho}^{2}g_{S}$ and is given by
  \begin{equation*}
    \tilde{\Box}_{\infty} = 2\pd[\bar{\tau}_+]\pd[\bar{\rho}] + (1 -
    2M\bar{\rho})(\bar{\rho}\pd[\bar{\rho}])^{2} + (1-4M\bar{\rho})\bar{\rho}\pd[\bar{\rho}] + \lap_{\omega}.
  \end{equation*}
  Note that because $v$ vanishes at $\bar{\rho} = 0$ for $\bar{\tau}_+ \leq
  -\bar{\tau}_{0}$, $\lap_{\omega}v$ vanishes there as well.  Thus at $\bar{\rho}
  = 0$, because $v$ is smooth as a function of $\bar{\rho}$, we have that
  \begin{equation*}
    2\pd[\bar{\tau}_+]\pd[\bar{\rho}]v|_{\bar{\rho}=0} =0,
  \end{equation*}
  for $\bar{\tau}_+ \leq -
  \bar{\tau}_{0}$.  Integrating in $\bar{\tau}_+$ shows that $\pd[\bar{\rho}]v$ vanishes
  there as well.

  Differentiating the equation in $\bar{\rho}$ yields
  \begin{equation*}
    2\pd[\bar{\tau}_+]\pd[\bar{\rho}]^{2}v + O(\bar{\rho})v = 0,
  \end{equation*}
  where $O(\bar{\rho})$ is the product of $\bar{\rho}$ and smooth differential
  operator in $\bar{\rho}$ and $\omega$.  In particular, we have that at
  $\bar{\rho} = 0$, $\pd[\bar{\tau}_+]\pd[\bar{\rho}]^{2}v = 0$ for $\bar{\tau} \leq -
  \bar{\tau}_{0}$.  Integrating again implies that $\pd[\bar{\rho}]^{2}v=0$
  there.  

  Proceeding inductively, after differentiating the equation $k$
  times, we may conclude that $\pd[\bar{\rho}]^{k+1}v =
  \pd[\bar{\tau}_+]\pd[\bar{\rho}]^{k+1}v=0$ for $\bar{\rho} =0$ and $\bar{\tau}_+ \leq
  -\bar{\tau}_{0}$.  In particular, $v$ vanishes to infinite order at
  $\bar{\rho}=0$ for $\bar{\tau}_+\leq -\bar{\tau}_{0}$.  

  A similar argument shows that if $\mathcal{R}_{\scrim}(\phi,\psi)$
  vanishes for $\bar{\tau}_- \geq \bar{\tau}_{0}$, then $v$ vanishes to infinite
  order there as well.  

  We now decompose the solution $u$ (and thus also the rescaled
  solution $v$) into spherical harmonics.  Letting $w_{j}$ be an
  eigenbasis of $L^{2}(\sphere^{2})$ with eigenvalues
  $-\lambda_{j}^{2}$, we write $u(t,r,\omega) = \sum
  u_{j}(t,r)w_{j}(\omega)$ and $v(\bar{\tau}, \bar{\rho}, \omega) = \sum
  v_{j}(\bar{\tau}, \bar{\rho}) w_{j}(\omega)$.  Note that $v_{j}$ and $u_{j}$ are
  related by a change of coordinates and the rescaling, i.e., $v_{j} =
  rw_{j}$.  This decomposition extends to the radiation field
  as well, so each $v_{j}$ vanishes to infinite order at $\bar{\rho}=0$ for
  $\bar{\tau}_{+} \leq - \bar{\tau}_{0}$ and $\bar{\tau}_{-} \geq \bar{\tau}_{0}$.  In
  $(\bar{\tau}_{+},\bar{\rho})$ coordinates, each function $v_{j}$ satisfies
  \begin{equation*}
    2\pd[\bar{\tau}_{+}]\pd[\bar{\rho}]v_{j} + (1-2M\bar{\rho})(\bar{\rho}\pd[\bar{\rho}])^{2}v +
    (1-4M\bar{\rho})\bar{\rho}\pd[\bar{\rho}]v - \lambda_{j}^{2}v + \gamma_{\infty} v = 0.
  \end{equation*}
  Similarly, in $(\bar{\tau}_-, \bar{\rho})$ coordinates, each $v_{j}$ satisfies
  \begin{equation*}
    -2\pd[\bar{\tau}_-]\pd[\bar{\rho}]v + (1-2M\bar{\rho})(\bar{\rho}\pd[\bar{\rho}])^{2}v +
    (1-4M\bar{\rho})\bar{\rho}\pd[\bar{\rho}]v - \lambda_{j}^{2}v + \gamma_{\infty} v = 0.
  \end{equation*}
  We may thus extend $v_{j}$ by $0$ to a smooth solution in a
  neighborhood of $\bar{\rho} =0$ for $\bar{\tau}_{+} \leq -\bar{\tau}_{0}$.  We may
  perform a similar extension at past null infinity as well.

  We now use an argument similar to the one in the proof of
  Theorem~\ref{thm:support-horiz} to conclude that $v$ (and hence $u$)
  must vanish in the region we claim.\footnote{One key difference
    between the two arguments is that we must proceed incrementally
    near null infinity because $u$ is not smooth as a function of
    $\bar{\tau}_{+}^{-1}$ and $\bar{\tau}_{-}^{-1}$.}  Suppose instead
  that the corresponding initial data for some $u_{j}$ is supported in
  $r \leq r_{1}$, corresponding to $\bar{\tau}_{1}$.  Near $\bar{\rho}
  = 0$ and $\bar{\tau}_{+} = -\bar{\tau}_{1}$, $v_{j}$ vanishes to
  infinite order at $\bar{\tau}_+ + \bar{\rho} = -\bar{\tau}_{1}$ and
  thus vanishes in a full neighborhood of $\bar{\rho} = 0$,
  $\bar{\tau}_{+} = -\bar{\tau}_{1}$.  A similar argument holds for
  $\bar{\rho} = 0$ and $\bar{\tau}_{-} = \bar{\tau}_{1}$.  Let us say
  that $v_{j}$ vanishes $\bar{\tau}_{+} \leq -\bar{\tau}_{1} + \delta$
  when $\bar{\rho} \leq \delta$ as well as for $\bar{\tau}_{-} \geq
  \bar{\tau}_{1} - \delta$ when $\bar{\rho} \leq \delta$.  In terms of
  $t$ and $\bar{\rho}$, this implies that $u_{j}$ vanishes if
  \begin{align*}
    r &\geq \delta^{-1} \\ \tau_{1} - \delta - r - 2M \log(r-2M)
    &\leq t \leq \delta - \tau_{1} + r + 2M\log(r-2M).
  \end{align*}
  We now use the hyperbolicity of the $1+1$-dimensional operator
  \begin{equation*}
     - \left( \frac{r}{r-2M}\right)\pd[t]^{2} + \left(
       \frac{r-2M}{r}\right)\pd[r]^{2} + \frac{2(r-M)}{r^{2}}\pd[r] - \frac{\lambda_{j}^{2}}{r^{2}}
   \end{equation*}
   with respect to $\pd[r]$ to conclude that the initial data for
   $u_{j}$ vanishes if $r + 2M\log (r-2M) \geq \bar{\tau}_{1}-\delta$,
   i.e., for $r \geq \tilde{r}_{1}$ with $\tilde{r}_{1} < r_{1}$.
   Figure~\ref{fig:regions-null-infinity} illustrates the process of
   improving from vanishing for $r \geq r_{1}$ to $r \geq
   \tilde{r}_{1}$.

   \begin{figure}[h]
     \centering
     \begin{tikzpicture}
       \draw [->] (3,0) -- (0,3) node[anchor=south]
       (tauplusaxis){$\bar{\tau}_{+}$};
       \draw [->] (3,0) -- (0,-3) node[anchor = north]
       (tauminusaxis){$-\bar{\tau}_{-}$};

       \fill[gray!20!white] (3,0) -- (2,1) node[black, anchor =
       south west]{$-\bar{\tau}_{1}$} -- (1,0) -- (2,-1) node[anchor=north
       west, black]{$\bar{\tau}_{1}$} -- (3,0);
       
       \draw [ultra thick] (3,0) -- (0.5,2.5)
       node[anchor=south](plustau0){$-\bar{\tau}_{0}$};
       \draw [ultra thick] (3,0) -- (0.5,-2.5) node[anchor=north] (minustau0)
       {$\bar{\tau}_{0}$}; 

       \draw (3,0) -- (2,1) -- (1,0) node[anchor=south]{$r_{1}$} -- (2,-1) -- (3,0);
       
       \draw (-2,0) node[anchor=south]{$t=0$} -- (3,0);

       \draw (2,1) -- (1.5,1.5) node [anchor = south
       west]{$-\bar{\tau}_{1} + \delta$} -- (1, 1) --
       (1.5,0.5) -- (2,1);
       \draw (2,-1) -- (1.5, -1.5) node [anchor = north west]
       {$\bar{\tau}_{1} - \delta$} -- (1,-1) -- (1.5, -0.5) -- (2,-1);

       \draw [dashed] (1.5,0.5) -- (2,0) -- (1.5, -0.5);

       \draw [dashed] (1,1) -- (0,0) node[anchor = south
       east]{$\tilde{r}_{1}$} -- (1,-1);

       \draw [fill=white](3,0) circle (0.1);

     \end{tikzpicture}
     \caption{Iteratively improving the support of the initial data near
       null infinity.  Knowing that the solution vanishes identically in
       the shaded region, we find first that it vanishes in the small
       white region, then use the hyperbolicity of the one-dimensional
       operator to improve it to the region outlined by the dashed line.}
     \label{fig:regions-null-infinity}
   \end{figure}

  Proceeding in this manner, we find that both $\phi$ and
  $\psi$ must be supported in $r \leq r_{0}$.
\end{proof}

 \bibliographystyle{alpha}
 \bibliography{papers}

\end{document}